\documentclass[12pt]{amsart}

\usepackage{multirow}
\usepackage{ulem}
\usepackage{enumerate}
\usepackage{dsfont}
\usepackage{amssymb,amsthm,amsmath}
\usepackage{mathrsfs}
\usepackage[colorlinks,breaklinks]{hyperref}
\usepackage[margin=0.6in]{geometry}
\usepackage{braket}
\usepackage{tikz}
\usetikzlibrary{positioning,arrows}
\usepackage{capt-of}
\usepackage{caption}
\usepackage{float}
\usetikzlibrary{shapes}
\usetikzlibrary{plotmarks}
\tikzset{
  state/.style={circle,draw,minimum size=6ex},
  arrow/.style={-latex, shorten >=1ex, shorten <=1ex}}

\theoremstyle{plain}
\newtheorem{prop}{Proposition}[section]
\newtheorem{lemma}[prop]{Lemma}
\newtheorem{thm}[prop]{Theorem}
\newtheorem{cor}[prop]{Corollary}

\theoremstyle{definition}

\theoremstyle{remark}
\newtheorem{remark}[prop]{Remark}

\newcommand{\HH}{\mathbb{H}}

\newcommand{\ZZ}{\mathbb{Z}}
\newcommand{\CC}{\mathbb{C}}

\newcommand{\sceq}{\mathrel{\mathop:}=}

\newcommand{\bmat}[4]{\begin{bmatrix} #1&#2\\#3&#4\end{bmatrix}}

\newcommand\numberthis{\addtocounter{equation}{1}\tag{\theequation}}

\DeclareMathOperator{\PSL}{PSL}
\DeclareMathOperator{\diag}{diag}

\DeclareMathOperator{\tr}{tr}

\DeclareMathOperator{\dvol}{dvol}
\DeclareMathOperator{\vol}{vol}

\DeclareMathOperator{\id}{id}

\allowdisplaybreaks
   
\begin{document}

\title[]{Hecke triangle groups, transfer operators and Hausdorff dimension}

\author[L.\@ Soares]{Louis Soares}
\email{louis.soares@gmx.ch}

\begin{abstract} 
We consider the family of Hecke triangle groups $ \Gamma_{w} = \langle S, T_w\rangle $ generated by the M\"obius transformations $ S : z\mapsto -1/z $ and $ T_{w} : z \mapsto z+w $ with $ w > 2.$ In this case the corresponding hyperbolic quotient $ \Gamma_{w}\backslash\mathbb{H}^2 $ is an infinite-area orbifold. Moreover, the limit set of $ \Gamma_w $ is a Cantor-like fractal whose Hausdorff dimension we denote by $ \delta(w). $ The first result of this paper asserts that the twisted Selberg zeta function $ Z_{\Gamma_{ w}}(s, \rho) $, where $ \rho : \Gamma_{w} \to \mathrm{U}(V) $ is an arbitrary finite-dimensional unitary representation, can be realized as the Fredholm determinant of a Mayer-type transfer operator. This result has a number of applications. We study the distribution of the zeros in the half-plane $\mathrm{Re}(s) > \frac{1}{2}$ of the Selberg zeta function of a special family of subgroups $( \Gamma_w^n )_{n\in \mathbb{N}} $ of $\Gamma_w$. These zeros correspond to the eigenvalues of the Laplacian on the associated hyperbolic surfaces $X_w^n = \Gamma_w^n \backslash \mathbb{H}^2$. We show that the classical Selberg zeta function $Z_{\Gamma_w}(s)$ can be approximated by determinants of finite matrices whose entries are explicitly given in terms of the Riemann zeta function. Moreover, we prove an asymptotic expansion for the Hausdorff dimension $\delta(w)$ as $w\to \infty$. 
\end{abstract}
\maketitle

\section{Introduction}
In \cite{Hecke} Hecke introduced the one-parameter family of subgroups $ \Gamma_{w} = \langle S, T_w\rangle $ of $ \mathrm{PSL}_2(\mathbb{R})= \mathrm{SL}_{2}(\mathbb{R})/\{ \pm\id\} $ generated the elements 
$$
T_w = \bmat{1}{w}{0}{1}\quad\text{and}\quad S =\bmat{0}{1}{1}{0},
$$
and their inverses, where $ w $ is a positive real number. On the hyperbolic plane 
$$
\HH^2 = \{ z = x+iy : x\in \mathbb{R},\, y\in \mathbb{R}_{>0}\},
$$
these elements act by the M\"{o}bius transformations $ S : z\mapsto -1/z $ and $ T_{w} : z \mapsto z+w. $ The groups $ \Gamma_{w} $, which came to be known as the `Hecke triangle groups',  naturally generalize the well-known modular group
$$
\mathrm{PSL}_{2}(\mathbb{Z}) = \left\{  \bmat{a}{b}{c}{d}\in \mathbb{Z}^{2 \times 2} : ad-bc=1 \right\},
$$
which corresponds to the case $ w = 1. $ Hecke showed that $ \Gamma_w $ is a Fuchsian group, that is, a discrete subgroup of $ \mathrm{PSL}_2(\mathbb{R}) $, if and only if $ w = 2\cos(\pi/q) $ for integer $ q\geq 3 $ or $ w \geq 2. $  Moreover, the set
\begin{equation}\label{fundamental_domain}
\mathcal{F}(w) = \left\{ z\in \HH^2 : \vert\mathrm{Re}(z)\vert < \frac{w}{2}, \; \vert z\vert > 1 \right\}
\end{equation}
provides a fundamental domain for the action of $ \Gamma_{w} $ on $ \mathbb{H}^2 $, see Figure \eqref{fundamental_domains_Hecke}. 

\begin{figure}[H]
\centering
\begin{tikzpicture}[xscale=1.2, yscale=1.2]
\fill (0,2) node[above] {$ w > 2 $};

\fill[fill=lightgray] (-3,0)  -- (3,0) -- (3,3) -- (-3,3);

\draw (2,0) arc (0:180:2);
\fill[fill=white] (2,0) arc (0:180:2);

\draw (-3,0) -- (-3,3);
\draw (3,0) -- (3,3);

\draw (-4,0) -- (4,0);

\filldraw (-3,0) circle (0.5pt) node[below] {$ -\frac{w}{2} $};
\filldraw (3,0) circle (0.5pt) node[below] {$ \frac{w}{2} $};

\filldraw (-2,0) circle (0.5pt) node[below] {$ -1 $};
\filldraw (2,0) circle (0.5pt) node[below] {$ 1 $};

\filldraw (0,0) circle (0.5pt) node[below] {$ 0 $};
\filldraw (0,2) circle (0.5pt) node[above] {$ i $};
\end{tikzpicture}
\caption{Fundamental domain $ \mathcal{F}(w) $ for $ \Gamma_{w} $ with $ w > 2 $}
\label{fundamental_domains_Hecke}
\end{figure}

In the present paper we will restrict our attention to the case $ w > 2$. In this case the quotient $ \Gamma_{w}\backslash\HH^2 $ is an \textit{infinite-area} hyperbolic orbifold with one cusp, one funnel and one conical singularity\footnote{the conical singularity is caused by the elliptic element $ S $ which fixes the point $ i $}. In particular, the limit set $ \Lambda(\Gamma_{w}) $ of $ \Gamma_{w} $ is a Cantor-like fractal whose Hausdorff dimension we denote by $ \delta(w) $. Equivalently, $ \delta(w) $ is the exponent of convergence of the Poincar\'{e} series for $ \Gamma_{w}, $ see \cite{Patterson}.

We are interested in the Selberg zeta function of $ \Gamma_{w} $ twisted by arbitrary finite-dimensional unitary representations $ \rho\colon \Gamma_{w} \to \mathrm{U}(V) $. It is defined for $ \mathrm{Re}(s) > \delta(w) $ by the infinite Euler product
\begin{equation}\label{defi_L_functions}
Z_{\Gamma_{w}}(s, \rho) = \prod_{[\gamma]} \prod_{k=0}^{\infty} \mathrm{det}_{V}\left( 1_V- \rho(\gamma)e^{-(s+k)\ell(\gamma)} \right),
\end{equation}
where $ [\gamma] $ runs over the conjugacy classes of primitive hyperbolic elements of $ \Gamma_{w} $ and $ \ell(\gamma) $ is the displacement length of $ \gamma $ (see Subsection \ref{standard_facts}). Notice that \eqref{defi_L_functions} reduces to the classical Selberg zeta function when $ \rho = \textbf{1} $ is the trivial one-dimensional representation.

Our first main result asserts that $ Z_{\Gamma_{w}}(s, \rho) $ can be realized as the Fredholm determinant of a well-chosen family of transfer operators.

\begin{thm}\label{TO_main_theorem}
Fix $ w > 2 $, let $ \Gamma_{w} = \langle S, T_w\rangle $ be the corresponding Hecke triangle group, and let $ \rho\colon \Gamma_{ w}\to \mathrm{U}(V) $ be a unitary representation with finite-dimensional representation space $ V. $ Let $ \mathbb{D} $ be the open unit disk of the complex plane and consider the operator $ \mathcal{L}_{s,w,\rho} $ acting on functions $ f\colon \mathbb{D}\to V $ via
\begin{equation}\label{TO_almost_Mayer}
\mathcal{L}_{s,w,\rho} f(z) = \sum_{n\in \ZZ \smallsetminus \{ 0\}} \gamma_n'(z)^{s} \rho(\gamma_n)^{-1} f\left(\gamma_n(z) \right), \quad z\in \mathbb{D},
\end{equation}
where $ \gamma_n \sceq S T_w^n. $ Then, for all $ s\in \mathbb{C} $ with $ \mathrm{Re}(s) >\frac{1}{2} $, equation \eqref{TO_almost_Mayer} defines a trace-class operator 
\begin{equation}\label{operator_in_thm}
\mathcal{L}_{s,w,\rho} \colon H^2(\mathbb{D}; V)\to H^2( \mathbb{D} ; V)
\end{equation}
(see Subsection \ref{TO_FS} for more details). Moreover, the twisted Selberg zeta function is represented by the Fredholm determinant of \eqref{operator_in_thm}, that is,
\begin{equation}\label{Fredholm_determinant}
Z_{\Gamma_{w}}(s, \rho) = \det\left( 1- \mathcal{L}_{s, w,\rho} \right)
\end{equation}
for all $ \mathrm{Re}(s) >\frac{1}{2}. $
\end{thm}

\begin{remark}
Identities such as \eqref{Fredholm_determinant} are well-known in thermodynamic formalism, a subject going back to Ruelle \cite{Ruelle_zeta}. The relation between the Selberg zeta function and transfer operators has been studied by a number of different authors. For the convex co-compact setting (no cusps) we refer to \cite{Pollicott, Pollicott_Rocha, Guillope_Lin_Zworski}. In the presence of cusps, the first example of an identity in the spirit of \eqref{Fredholm_determinant} was given by Mayer \cite{Mayer_thermoPSL} for the modular group $ \Gamma_1 = \PSL_2(\ZZ) $ and for the trivial twist $ \rho = \textbf{1} $. Hecke triangle groups (cofinite and non-cofinite) have been studied extensively in Pohl \cite{Pohl_representation, Pohl_hecke_infinite}, where a version of \eqref{Fredholm_determinant} has been proven by geometrical methods and using different transfer operators. Our proof relies solely on certain combinatorial features of the group $ \Gamma_{w} $ and is reminiscent of the method of Lewis--Zagier \cite{Lewis-Meyer-two} for the modular group. Related work includes \cite{Fried_triangle, Morita_transfer, Moeller_Pohl, Mayer_Muehlenbruch_Stroemberg, FP_szf}.
\end{remark}

The representation of the Selberg zeta functions in terms of transfer operators has proven to be a powerful tool in the spectral theory of infinite-area hyperbolic surfaces, a subject not yet fully explored. For instance, transfer operator techniques have been implemented in \cite{JNS} to construct hyperbolic surfaces with arbitrarily small `spectral gap'. In \cite{Guillope_Lin_Zworski, Naud_Pohl_Soares}, transfer operators have been used to prove fractal Weyl bounds for resonances of the Laplacian on hyperbolic surfaces, analogous to Sj\"ostrands pioneering work \cite{Sjoestrand} on semi-classical Schr\"odinger operators. Related works where thermodynamic formalism plays an essential role include \cite{Naud_resonancefree,Naud_inventiones,Naud_crit,JN,BGS, OhWinter}. Another application of Fredholm determinant identities such as \eqref{Fredholm_determinant} is a simple proof of meromorphic continuation of the twisted Selberg zeta function, which is far from obvious from its definition in \eqref{defi_L_functions} as an infinite product over primitive conjugacy classes. Theorem \ref{TO_main_theorem} gives a new proof of the following result:

\begin{cor}\label{main_corollary}
Assumptions being as in Theorem \ref{TO_main_theorem}, the Selberg zeta function $ Z_{\Gamma_{w}}(s, \rho) $ admits a meromorphic continuation to $ s\in \mathbb{C} $ and all its poles are contained in $ \frac{1}{2}(1-\mathbb{N}_0). $
\end{cor}
 
In this paper we give additional applications of Theorem \ref{TO_main_theorem}. The transfer operator obtained in Theorem \ref{TO_main_theorem} can be used to study the Hausdorff dimension $ \delta(w) $ of the limit set $ \Lambda(\Gamma_{w}). $ Apart from its intrinsic interest, the Hausdorff dimension of the limit set of Fuchsian groups plays a profound role in the spectral theory of hyperbolic surfaces. For instance, the base eigenvalue of the Laplacian on $ \Gamma_{w} \backslash \HH^2 $ is known to be equal to $ \delta(w) (1-\delta(w)) $ by Patterson's result \cite[Theorem~7.2]{Patterson}. 

For applications to spectral theory, it is sometimes more convenient to work with \textit{torsion-free} Fuchsian groups $\Gamma$ in which case the quotient $\Gamma\backslash\mathbb{H}^2$ is a \textit{smooth} surface\footnote{that is, $X_w^0$ has no conical singularities}. Selberg's lemma \cite{SelbergTata} says that every finitely generated Fuchsian group has a finite-index, torsion-free subgroup. In the case of Hecke triangle groups there is a simple way of manufacturing such a subgroup $\Gamma_w^1 \subset \Gamma_w $. Indeed, let $ \rho \colon \Gamma_{w}\to \mathbb{C}^{\times} $ be the one-dimensional representation defined by $ \rho(T_{w}) = 1 $ and $ \rho(S) = -1 $, and set $ \Gamma_{w}^{1} = \ker(\rho) $. The group $ \Gamma_{w}^{1} $ is a normal subgroup of $ \Gamma_{w} $ (being the kernel of a homomorphism) and it is freely generated\footnote{this means that there are no relations between the generators $T_w$ and $R_w$ except for the trivial relations of the form $\gamma^{-1}\gamma = \gamma \gamma^{-1} = \id$} by the elements
\begin{equation}\label{generators_for_Gamma_0}
T_w^\pm = \bmat{1}{\pm w}{0}{1}\quad\text{and}\quad R_w^{\pm} \sceq S T_{w}^\pm S =\bmat{1}{0}{\mp w}{1}.
\end{equation}
In particular $ \Gamma_{w}^{1} $ contains no elliptic elements and it is therefore torsion-free. Moreover we have
$$
\Gamma_{w}/\Gamma_{w}^{1} \simeq \{ \mathrm{id}, S \} \simeq \mathbb{Z}/2\mathbb{Z},
$$
so the action of $ \Gamma_{w}^{1} $ on $ \HH^2 $ has the fundamental domain
\begin{equation}\label{fundamental_domain_1}
\mathcal{F}^{1}(w) = \mathcal{F}(w) \cup S .\mathcal{F}(w), 
\end{equation}
where $\mathcal{F}(w)$ is the fundamental domain of $\Gamma_w$ given in \eqref{fundamental_domain}, see Figure \ref{fundamental_domains_Hecke_1}.

\begin{figure}[H]
\centering
\begin{tikzpicture}[xscale=1.2, yscale=1.2]
\fill (0,2) node[above] {$ w > 2 $};

\fill[fill=lightgray] (-3,0)  -- (3,0) -- (3,3) -- (-3,3);

\draw[dashed] (2,0) arc (0:180:2);

\draw[black] (0,0) arc (0:180:0.66);
\draw[black] (1.32,0) arc (0:180:0.66);

\fill[fill=white] (0,0) arc (0:180:0.66);
\fill[fill=white] (1.32,0) arc (0:180:0.66);

\draw (-3,0) -- (-3,3);
\draw (3,0) -- (3,3);

\draw (-4,0) -- (4,0);

\filldraw (-3,0) circle (0.5pt) node[below] {$ -\frac{w}{2} $};
\filldraw (3,0) circle (0.5pt) node[below] {$ \frac{w}{2} $};
\filldraw (1.32,0) circle (0.5pt) node[below] {$ \frac{2}{w} $};
\filldraw (-1.32,0) circle (0.5pt) node[below] {$ -\frac{2}{w} $};
\filldraw (0,0) circle (0.5pt) node[below] {$ 0 $};
\filldraw (0,2) circle (0.5pt) node[above] {$ i $};
\end{tikzpicture}
\caption{Fundamental domain $ \mathcal{F}^{1}(w) $ for $ \Gamma_{w}^1 $ with $ w > 2 $}
\label{fundamental_domains_Hecke_1}
\end{figure}

It follows that the associated hyperbolic quotient $X_w^1= \Gamma_w^1\backslash\mathbb{H}^2$ is a smooth 2-cover of $X_w$. More generally, for every positive integer $n$, we can define a family of torsion-free subgroups 
$$ \Gamma_w^n = \ker( \rho_n ) $$ 
as the kernel of the representation $ \rho_n\colon \Gamma_w \to\mathbb{C}^\times$ given by
$$
\rho_n (S) = -1 \quad \text{and} \quad \rho(T_w) = e^{\frac{2\pi i}{n}}.
$$ 
The corresponding quotients $ X_w^n = \Gamma_w^n\backslash \mathbb{H}^2 $ are simultaneous covers of both $X_w$ and $X_w^1$. The associated covering groups can be shown to be isomorphic to 
$$
\Gamma_w/\Gamma_w^n \simeq \mathbb{Z}/2n\mathbb{Z} \quad \text{and} \quad \Gamma_w^1/\Gamma_w^n \simeq \mathbb{Z}/n\mathbb{Z},
$$
respectively. In particular, $X_w^n$ is a smooth, abelian $2n$-covering of the Hecke orbifold $X_w = \Gamma_w \backslash \mathbb{H}^2$. We can now formulate our next theorem.

\begin{thm}\label{zeros_and_eigenvalues}
Let $w>2$. Then
\begin{enumerate}[{\rm (i)}]
\setlength\itemsep{1em}
\item \label{Part_1} the Selberg zeta function $Z_{\Gamma_w}(s)$ has exactly one zero in the half-plane $\mathrm{Re}(s) > \frac{1}{2}$, namely at $s=\delta(w)$,

\item\label{Part_2} for every positive integer $n$ the Selberg zeta function $Z_{\Gamma^n_w}(s)$ has at most $n$ zeros in the half-plane $\mathrm{Re}(s) > \frac{1}{2}$. In particular, the number of $L^2$-eigenvalues of the positive Laplacian on $X_w^n$ is at most $n$, and

\item\label{Part_3} for every $\varepsilon > 0$ there exists a constant $c=c(\varepsilon,w)>0$ such that for every $n$ the Selberg zeta function $ Z_{\Gamma_w^n}(s) $ has at least $c n$ zeros in the interval 
$$ (\delta(w) - \varepsilon, \delta(w)]. $$
In particular, for every $\varepsilon' > 0$ there exists $c'= c'(\varepsilon',w) > 0$ such that the positive Laplacian on $X_w^n$ has at least $c' n$ $L^2$-eigenvalues in 
$$ [\lambda_0(w), \lambda_0(w)+\varepsilon') $$ 
where $\lambda_0(w) = \delta(w)(1-\delta(w))$ is the common base eigenvalue of the surfaces $X_w^n.$  
\end{enumerate}
\end{thm}

\begin{remark}
Part \eqref{Part_1} of Theorem \ref{zeros_and_eigenvalues} should be compared with \cite[Theorem~6.1]{Phillips_Sarnak_Hecke_1}, which states that for all $ w > 2 $ the base eigenvalue $ \delta(w)(1-\delta(w)) $ is the \textit{only} Laplace eigenvalue for the Hecke orbifold $ X_w $. From Borthwick--Judge--Perry \cite{BJP} we know that if $ \Gamma $ is a finitely generated, \textit{torsion-free} Fuchsian group, then the zeros of the Selberg zeta function $ Z_{\Gamma}(s) $ in the half-plane $ \mathrm{Re}(s) >\frac{1}{2} $ correspond to the $ L^2 $-eigenvalues $ s(1-s) $ of the Laplacian on $\Gamma\backslash \HH^2 $. Thus \cite[Theorem~6.1]{Phillips_Sarnak_Hecke_1} is morally equivalent to Part \eqref{Part_1} of Theorem \ref{zeros_and_eigenvalues}. Unfortunately, the result of Borthwick--Judge--Perry does not apply directly to any of the groups $ \Gamma_{w} $, since they contain the elliptic element $ S $. Nevertheless, one should expect the zeros of $ Z_{\Gamma_{w}}(s) $ to have a similar interpretation in terms of eigenvalues of the Laplacian (though the author is not aware of such a result in the literature).
\end{remark}

\begin{remark}
Part \eqref{Part_3} of Theorem \ref{zeros_and_eigenvalues} says that on large abelian covers of $X_w$, the Laplacian possesses a large number of eigenvalues arbitrarily close to base eigenvalue $\lambda_0(w).$ Similar results were proven for the modular surface $X_1$ by Selberg \cite[paper~33, p.~12]{Selberg_collected_papers} and for compact hyperbolic surfaces by Randol \cite{Randol}, both using completely different methods. More recently, using transfer operator techniques, a similar (and more precise) result was established for convex co-compact surfaces by Jakobson, Naud and the author in \cite{JNS}.
\end{remark}

The next result shows that the classical Selberg zeta function $Z_{\Gamma_w}(s)$ can be approximated by determinants of $k\times k$-matrices, up to an error that tends to zero exponentially fast as $k\to \infty$. More concretely, we have

\begin{thm}\label{approximation_by_matrices}
For all $w>2$ and $\mathrm{Re}(s) > \frac{1}{2}$ we have
$$
\left\vert Z_{\Gamma_w}(s) - D_k(s,w) \right\vert \leq C \left( \frac{w}{2} \right)^{-k + o(k)}
$$
where $C = C(s,w) > 0$ is some constant independent of $k$ and $D_k(s,w) $ is the determinant 
\begin{equation}\label{determinant_in_theorem}
D_k(s,w) = \det (1-A_k(s,w))
\end{equation}
where $A_k(s,w) = \left( a_{i,j}(s,w) \right)_{0\leq i,j < k}$ is the matrix given by
$$
a_{i,j}(s,w) = \left( (-1)^{i+j} + 1 \right) \frac{\zeta(2s+i+j)}{w^{2s+i+j}} { 2s+i+j-1 \choose i}.
$$
Here $\zeta$ denotes the Riemann zeta function. Moreover, for $\varepsilon > 0$ sufficiently small $D_k(s,w)$ has precisely one zero $s_k(w)$ in the half-plane $\mathrm{Re}(s) \geq \frac{1}{2}+\varepsilon $ for all $k$ sufficiently large and we have
$$
\lim_{k\to \infty} s_k(w) = \delta(w).
$$
\end{thm}

\begin{remark}
Jenkinson--Pollicott \cite{Jenkinson-Pollicott} proposed an algorithm to numerically compute the Hausdorff dimension for limit sets of certain Kleinian groups, using in a fundamental way transfer operators $ \mathcal{L}_{s} $ associated to these sets. In the setting of \cite{Jenkinson-Pollicott}, $ \mathcal{L}_{s} $ is always given by a \textit{finite} sum of composition operators. The transfer operator of Theorem \ref{TO_main_theorem} is an infinite sum of composition operators, making the analysis of Jenkinson--Pollicott more complicated for the task of estimating the Hausdorff dimension for Hecke triangle groups. Theorem \ref{approximation_by_matrices} provides a different method to compute $\delta(w).$ For any given $w > 2$ and $k$ sufficiently large, the numbers $s_k(w)$ can be calculated with arbitrary precision using a computer. Since we have made no attempt to precisely estimate the error $ \vert \delta(w) - s_k(w)\vert $, the values of $s_k(w)$ yield only empirical estimates for $\delta(w)$. Nevertheless, these values are in perfect agreement with the approximations given by Phillips and Sarnak in \cite{Phillips_Sarnak_Hecke_groups}:

\begin{table}[H]
  \begin{center}
    \label{tab:table1}
    \begin{tabular}{l|c|r}
	 &Approximations for $\delta(w)$ & Approximations for $s_{15}(w)$  \\
	 $ w $ &of Phillips and Sarnak & from Theorem \ref{approximation_by_matrices}\\
      \hline
      $2.5$ & $0.816 \pm 0.002$ & $ 0.82 $\\
      \hline
      $3$ 	& $0.753 \pm 0.003$ & $ 0.752 $ \\
      \hline
      $4$	& $0.683 \pm 0.005$ & $ 0.6837 $\\
      \hline
      $6$	& $0.621 \pm 0.001$ & $ 0.622970 $\\
      \hline
      $8$ 	& $0.595 \pm 0.004$ & $ 0.593957 $\\
      \hline
      $10$ 	& $0.575 \pm 0.007$ & $ 0.5766067$\\
      \hline
      $16$ 	& $0.550 \pm 0.005$ & $ 0.5501100 $\\
      \hline
      $40$ 	& $0.520 \pm 0.007$ & $ 0.521821511 $\\
      \hline
      $100$ & $0.509 \pm 0.002$ & $ 0.509279417381 $\\
    \end{tabular}
  \end{center}
\end{table}
\end{remark}

The properties of Hausdorff dimension $ \delta(w) $ have been studied by several authors \cite{Beardon_exponent1, Patterson, Pignataro_Thea, Phillips_Sarnak_Hecke_1}. It is known from these papers that
\begin{equation}\label{prelim_results}
\delta(w) > \frac{1}{2}, \quad \delta(2) = 1, \quad \lim_{w\to \infty}\delta(w) = \frac{1}{2},
\end{equation}
and that $ w\mapsto \delta(w) $ is a strictly decreasing Lipschitz continuous function on $ [2, \infty) $. In addition, Phillips--Sarnak \cite{Phillips_Sarnak_Hecke_groups} proved that the base eigenvalue $ \delta(w)(1-\delta(w)) $ is analytic and concave as a function of $ w\in [2, \infty) $. Our next result is the following

\begin{thm}\label{thm:asymptotic_expansion}
As $ w\to \infty $ we have the asymptotic expansion
\begin{equation}\label{asympt:eqn}
\delta(w) = \frac{1}{2} + \frac{1}{w} -  \frac{2\log w}{w^2} +  \frac{2\gamma_0}{w^2} + \sum_{j=2}^{4} \frac{P_j(\log w)}{w^{j+1}}  +  O\left( \frac{(\log w)^5}{w^{6}} \right).
\end{equation}
Here, the error term does not depend on $ w $, $ \gamma_0 \approx 0.5772156649 $ is the Euler--Mascheroni constant, and each $ P_j $ ($ j=2,3,4 $) is a polynomial of degree $ j $ whose coefficients can be computed explicitly in terms of the Stieltjes constants.
\end{thm}

\begin{remark}
It is likely that our proof method can be extended to give an asymptotic expansion with more terms on the right hand side of \eqref{asympt:eqn}.
\end{remark}

\begin{remark}\label{Estimating_delta_3}
Although Theorem \ref{thm:asymptotic_expansion} is concerned with the asymptotic behaviour of $ \delta(w) $ as $ w\to \infty $, the methods developed to prove it may also be used to give numerical estimates for small $ w.$ As a concrete example, we estimate the value $ \delta(3) $ to be in the range
$$
0.75065 < \delta(3) < 0.75322,
$$
see Subsection \ref{Sharp numerical estimates}. This sharpens the estimate of Phillips--Sarnak in \cite{Phillips_Sarnak_Hecke_groups} and it answers in the affirmative a question posed by Jakobson--Naud \cite{JN_lower_bounds} whether the quantity $ \delta(3) $ is strictly larger than $ \frac{3}{4}. $
\end{remark}

\begin{remark}
An asymptotic formula similar to the one in Theorem \ref{thm:asymptotic_expansion} was proved by Hensley \cite{Hensley} for the Hausdorff dimension of the set $ E_n $ as $ n\to \infty $, where $ E_n $ consists of all reals $ x\in (0,1) $ for which the infinite continued fraction 
$$
x = \frac{1}{a_1 + \frac{1}{a_2 + \frac{1}{\cdots}}}
$$
has all its partial quotients $ a_j $ in $ \{ 1,\dots, n\}$. Similar asymptotic formulas for the Julia set related to the quadratic map $f_c(x) = x^2+c$ appear in \cite{Bodart_Zinsmeister,RuelleRepellers}. 
\end{remark}

\paragraph{\textbf{Notation}}
We write $ f(x) = O(g(x)) $ and $ f(x) = o(g(x)) $ as $x\to a$ to mean $\limsup_{x\to a} \vert f(x)/g(x)\vert < \infty $ and $ \lim_{x\to a} f(x)/g(x)  = 0 $ respectively. We use the symbol $ f(x) \ll g(x) $ to mean $ f(x) \leq C g(x) $ for some implied constant $C > 0$ not depending on $x$.

\paragraph{\textbf{Organization}}
In Section \ref{Section_1} we begin by briefly recalling a few facts on hyperbolic geometry and singular values needed in this paper. After having precisely defined the transfer operator in Subsection \ref{TO_FS} and the function space on which it acts, we prove Theorem \ref{TO_main_theorem} and Corollary \ref{main_corollary}. In Section \ref{Section_2} we prove Theorem \ref{zeros_and_eigenvalues} and in Section \ref{Section_3} we prove Theorem \ref{approximation_by_matrices} and Theorem \ref{thm:asymptotic_expansion}.

\section{Twisted Selberg zeta function and transfer operators}\label{Section_1}

\subsection{Hyperbolic geometry}\label{standard_facts} For a thorough discussion on hyperbolic surfaces, Fuchsian groups (of finite and infinite covolume) and their spectral theory, we refer to Borthwick's book \cite{Borthwick_book}. One of the standard models for the hyperbolic plane is the Poincar\'{e} half-plane
$$
\HH^2 = \{ z = x+iy : x\in \mathbb{R},\, y\in \mathbb{R}_{>0}\}, \quad
ds^{2} = \frac{dx^{2}+dy^{2}}{y^{2}}.
$$
The group of orientation-preserving isometries of $ (\mathbb{H}^2, ds) $ is isomorphic to
$$
\mathrm{PSL}_{2}(\mathbb{R}) = \left\{  \bmat{a}{b}{c}{d}\in \mathbb{R}^{2 \times 2} : ad-bc=1 \right\}.
$$
The elements of this group act on $ \mathbb{H}^2 $ by M\"obius transformations:
$$
\gamma = \bmat{a}{b}{c}{d} \in \mathrm{PSL}_2(\mathbb{R}),\, z\in \mathbb{H}^2 \quad \Longrightarrow \quad  \gamma(z) \sceq \frac{az+b}{cz+d}.
$$
This action extends continuously to the boundary $ \partial \HH^2 = \mathbb{R}\cup \{ \infty \} $ and to the whole Riemann sphere $ \overline{\CC} $. Now let $ \Gamma < \PSL_2(\mathbb{R}) $ be a discrete\footnote{`discrete' with respect to the matrix topology on $ \PSL_2(\mathbb{R}) $ defined by the norm $ \Vert A\Vert = \sqrt{\mathrm{\tr}(A^\ast A)} $} and finitely generated group. The limit set $ \Lambda(\Gamma) $ of $ \Gamma $ is defined as the set of accumulation points (in the Riemann sphere topology) of all orbits $ \Gamma.z = \{ \gamma(z) : \gamma\in \Gamma \}. $ It turns out that the quotient $ \Gamma\backslash\mathbb{H}^2 $ has infinite hyperbolic volume if and only if $ \Lambda(\Gamma) $ is a perfect, nowhere dense subset of $ \partial\mathbb{H}^2. $

An element $ \gamma\in \Gamma $ is said to be \textit{primitive} if is not a proper power $ \widehat{\gamma}^k $ of some element $ \widehat{\gamma} \neq \gamma. $ An element $ \gamma\in \Gamma $ is said to be \textit{hyperbolic} if its action on $ \mathbb{H}^2 $ has two distinct fixed points on $ \partial \mathbb{H}^2 $, or equivalently, if $ \vert \tr\gamma\vert > 2. $ Every hyperbolic transformation $\gamma$ is conjugate to the map $ z\mapsto e^{\ell} z $ where $\ell = \ell(\gamma) \in \mathbb{R} $, called the \textit{displacement length}, is given by the formula
\begin{equation}\label{displacement}
2 \cosh\left( \frac{\ell(\gamma)}{2} \right) = \vert \tr \gamma\vert.
\end{equation}
Notice that $ \vert \tr \gamma\vert $ is well-defined in $ \mathrm{PSL}_2(\mathbb{R}). $ \eqref{displacement} reveals that $ \ell(\gamma) $ is invariant under conjugations, since the trace is. In particular the displacement length is constant on each $\Gamma$-conjugacy class 
$$ [\gamma] \sceq \{ g\gamma g^{-1} : g\in \Gamma  \}. $$
We denote by $ [\Gamma]_{h} $ the set of conjugacy classes of hyperbolic elements of $ \Gamma $ and we denote by $ [\Gamma]_{p} $ the set of conjugacy classes of primitive hyperbolic elements of $ \Gamma $. It is well know that the set of closed primitive geodesics on $ \Gamma\backslash\mathbb{H}^2 $ is bijective to the set $ [\Gamma]_p $. Moreover, given a conjugacy class $ [\gamma]\in [\Gamma]_p $, the length of the corresponding geodesic is equal to the displacement length $ \ell(\gamma). $ 
 
Since $\ell(\gamma)$ is constant on each conjugacy class $ [\gamma]\in [\Gamma]_h $, the Euler product definition of the twisted Selberg zeta function 
\begin{equation}\label{defi_L_functions_2}
Z_{\Gamma}(s, \rho) = \prod_{[\gamma]\in [\Gamma]_{p}} \prod_{k=0}^{\infty} \mathrm{det}_{V}\left( 1_V- \rho(\gamma)e^{-(s+k)\ell(\gamma)} \right),
\end{equation}
is independent of the choice of the representative of each conjugacy class $ [\gamma]. $ Here, $\rho \colon \Gamma_w \to \mathrm{U}(V)$ is assumed to be a unitary representation of the group $\Gamma_w$ with finite-dimensional representation space $V$.

Let us explain why the right hand side of \eqref{defi_L_functions_2} converges in the half-plane $\mathrm{Re}(s) > \delta$, where $\delta$ denotes the Hausdorff dimension of the limit set $\Lambda(\Gamma).$ In view of the prime geodesic theorem (see \cite[Chapter~14]{Borthwick_book} and references therein) we may redefine the quantity $\delta$ as the abscissa of convergence of the series
$$
\sum_{[\gamma] \in [\Gamma]_{p}} e^{ - s \ell(\gamma)},
$$
that is,
\begin{equation}\label{convergence_abscissa}
\sum_{[\gamma] \in [\Gamma]_{p}} e^{ -s\ell(\gamma)} < \infty \Longleftrightarrow \mathrm{Re}(s) > \delta.
\end{equation}
Since $\rho$ is assumed to be a unitary representation, the eigenvalues of $\rho(\gamma)$ lie on the unit circle for every $\gamma$, showing that
\begin{equation}\label{bound_for_det}
\mathrm{det}_{V}\left( 1_V- \rho(\gamma)e^{-(s+k)\ell(\gamma)} \right) \leq \left( 1+e^{-(s+k)\ell(\gamma)} \right)^{\dim(\rho)} \leq \exp\left(  \dim(\rho) e^{-(s+k)\ell(\gamma)}  \right)
\end{equation}
where $\dim(\rho) \sceq \dim(V)$ denotes the dimension of $\rho$. Combining \eqref{convergence_abscissa} and \eqref{bound_for_det} shows that the product on left-hand side of \eqref{defi_L_functions_2} converges in the half-plane $\mathrm{Re}(s) > \delta$. 

\subsection{Singular values and Fredholm determinants}
In this subsection we collect some preliminaries about singular values which will be used repeatedly in this paper. Good references for the general theory of singular values and Fredholm determinants include \cite{Gohberg_Krein, Gohberg_Goldberg_Krupnik, Simon}. 

Given two separable Hilbert spaces $\mathcal{H}_1$ and $\mathcal{H}_2$ and a compact operator $\mathcal{A}\colon \mathcal{H}_1\to \mathcal{H}_2$ we let $\mathcal{A}^\ast \colon \mathcal{H}_2\to \mathcal{H}_1$ denote its adjoint operator. Note that $\mathcal{A}^\ast\mathcal{A} \colon \mathcal{H}_1\to \mathcal{H}_1$ is a positive and symmetric operator. The absolute value of $\mathcal{A}$, denoted by $\vert \mathcal{A}\vert$, is the unique positive and symmetric operator $\mathcal{H}_1\to \mathcal{H}_1$ satisfying $\vert \mathcal{A}\vert^2 = \mathcal{A}^\ast \mathcal{A}$. The \textit{singular values} of $\mathcal{A}$ are the nonzero eigenvalues of $\vert \mathcal{A}\vert$, arranged in decreasing order,
$$
\mu_1(A) \geq \mu_2(A) \geq \cdots .
$$
If necessary, we turn this sequence into an infinite one by filling it up with zeros at the end. We say that $\mathcal{A}$ is a trace-class operator if 
$$
\Vert \mathcal{A}\Vert_1 \sceq \sum_{k=1}^\infty \mu_k(\mathcal{A}) < \infty.
$$ 
It is well-known that $\Vert \cdot \Vert_1$ is a norm, called the \textit{trace norm}. The min-max characterization of singular values says that 
\begin{equation}\label{min_max}
\mu_m(A) = \min_{\substack{ V \subset \mathcal{H} \\ \dim(V) = m-1 }} \max_{ \psi\in V^{\perp} } \frac{\Vert A \psi\Vert}{\Vert \psi\Vert},
\end{equation} 
where the minimum is taken over all $m-1$-dimensional subspaces of $\mathcal{H}$. It follows immediately that the largest singular value is equal to the operator norm:
$$
\mu_1(A) = \Vert A\Vert.
$$
The min-max characterization can also be used to derive the following estimate: for any given orthonormal basis $\{ \psi_m \}_{m\in \mathbb{N}_0}$ of $\mathcal{H}$ we have
\begin{equation}\label{min_max_consequence}
\mu_n(\mathcal{A}) \leq \sum_{k\geq n} \Vert \mathcal{A}\psi_k\Vert.
\end{equation}
Now, for every trace-class operator $\mathcal{A}\colon \mathcal{H} \to\mathcal{H} $ and for every $u\in \mathbb{C}$ sufficiently small we have the absolutely convergent expansion for the Fredholm determinant
\begin{equation}\label{Fredholm_expansion_0}
\det\left( 1- u \mathcal{A} \right) = \exp\left( \tr \log\left( 1-u \mathcal{A} \right) \right) = \exp\left( - \sum_{N=1}^{\infty} \frac{u^{N}}{N} \mathrm{tr}\left( \mathcal{A}^{N} \right) \right).
\end{equation}
This is a direct consequence of Lidskii's theorem, see \cite[Chapter~3]{Simon}.

Let us conclude this subsection with an estimate for Fredholm determinants which proves extremely useful in this paper: if both $\mathcal{A}$ and $\mathcal{B}$ are trace-class operators, then
\begin{equation}\label{Simon_continuity_bound}
\left\vert \det\left( 1-\mathcal{A} \right) - \det\left( 1-\mathcal{B} \right) \right\vert \leq  \Vert \mathcal{A} - \mathcal{B} \Vert_1 \exp\left( \Vert \mathcal{A} \Vert_1 + \Vert \mathcal{B} \Vert_1 + 1 \right),
\end{equation}
see for instance \cite[Corollary~4.2]{Gohberg_Goldberg_Krupnik}.

\subsection{Transfer operator and function space}\label{TO_FS}
Recall that the Hecke triangle group $ \Gamma_{w} $ is defined to be the subgroup of $ \mathrm{PSL}_2(\mathbb{R}) $ generated by the two elements
\[
T_w \sceq \bmat{1}{w}{0}{1}\quad\text{and}\quad S\sceq \bmat{0}{1}{-1}{0}.
\]
We will henceforth assume that $ w > 2  $ in which case $ \Gamma_{w} $ is a Fuchsian group with infinite co-volume, i.e., the hyperbolic quotient $ \Gamma_w\backslash \mathbb{H}^2 $ has infinite area.

From now on $ V $ is a finite-dimensional complex vector space endowed with the hermitian inner product $ \langle\cdot, \cdot\rangle_V $ and $ \rho \colon \Gamma_{w} \to \mathrm{U}(V) $ is a unitary representation of $ \Gamma_{w} $. 

Let $ \mathbb{D} = \{ \vert z\vert < 1 \} $ be the open unit disk in the complex plane. The function space of interest is the \textit{vector-valued Bergman space}
\begin{equation}\label{vector_valued_Bergmann}
H^2(\mathbb{D} ; V) \sceq \left\{ \text{$f\colon \mathbb{D} \to V$ holomorphic} \ \left\vert\ \Vert f\Vert < \infty \right.\right\},
\end{equation}
with $ L^{2} $-norm given by
$$
\Vert f\Vert^2 \sceq \int_{\mathbb{D}} \Vert f(z)\Vert_{V}^{2}\dvol(z).
$$
Here $ \vol $ denotes the Lebesgue measure and $ \Vert \cdot\Vert_V $ is the norm on $ V $ induced by $ \langle\cdot, \cdot\rangle_V $. Endowed with the inner product
\[
 \langle f, g\rangle \sceq \int_{\mathbb{D}}  \langle f(z), g(z)\rangle_V\dvol(z),
\]
the space $ H^{2}(\mathbb{D} ; V) $ is a Hilbert space. Notice that $ H^{2}( \mathbb{D} ; \CC) = H^{2}(\mathbb{D} ) $ is the classical Bergman space over $ \mathbb{D} $. 

Now for every $ n\in \mathbb{Z} $ we define the element
$$ 
\gamma_{n} \sceq S T_{w}^{n} = \bmat{0}{-1}{1}{nw}.
$$
Note that $\gamma_n$ is hyperbolic for all $ n\in \mathbb{Z}\smallsetminus \{ 0\} $ since $w>2$. Finally, we define the (initially only formal) transfer operator
\begin{equation}\label{definition_of_TO}
\mathcal{L}_{s, w, \rho} f(z) = \sum_{n\in \ZZ \smallsetminus \{ 0\}} \gamma_n'(z)^{s} \rho(\gamma_n)^{-1} f\left(\gamma_n(z) \right), \quad z\in \mathbb{D},
\end{equation}
acting on functions $ f\in H^2(\mathbb{D} ; V). $ Notice that the M\"{o}bius transformation $ \gamma_n $ and its derivative are given by 
$$
\gamma_n(z) = -\frac{1}{z+nw} \quad \text{and}\quad \gamma_n'(z) = \frac{1}{(z+nw)^2}.
$$
In particular, since $ w > 2 $, for all $n\neq 0$ the derivative $ \gamma_n' $ is positive on the interval $ [-1, 1] $ and non-zero in the disk $ \mathbb{D} $. The complex powers $ \gamma_n'(z)^s $ make sense for all $ z\in \mathbb{D} $ by writing
\begin{equation}\label{complex_power}
\gamma_n'(z)^{s} = \left( \vert n\vert w \right)^{-2s} e^{-2s \log \left( 1+\frac{z}{nw} \right)}
\end{equation}
Here the logarithm is given by the usual Taylor-expansion
$$
\log(1+u) = u - \frac{u^2}{2} + \frac{u^3}{3}\pm \cdots,
$$
which is valid for all $ \vert u\vert < 1 $. Hence, the right hand side of \eqref{complex_power} well-defined for all $ z\in \mathbb{D} $ and all $ n\in \mathbb{Z}\smallsetminus \{ 0\} $.

Note that $ \mathcal{L}_{s,w,\rho} $ can be written as the infinite sum 
\begin{equation}\label{sum_composition_operator}
\mathcal{L}_{s, w,\rho} = \sum_{n\in \mathbb{Z}\smallsetminus \{0\}} \nu_{s,\rho}(\gamma_n^{-1}) = \sum_{n\in \mathbb{Z}\smallsetminus \{0\}} \nu_{s,\rho}(T_w^{-n}S) =  \sum_{n\in \mathbb{Z}\smallsetminus \{0\}} \nu_{s,\rho}(T_{w}^{n}S).
\end{equation}
where for every element $\gamma$, the $\nu_{s,\rho}(\gamma)$'s are composition operators of the form
\begin{equation}\label{compostion_operators}
\nu_{s,\rho}(\gamma)\colon H^2( \mathbb{D}; V )\to H^2(\mathbb{D}; V), \quad \nu_{s,\rho}(\gamma)f(z) := \left[ (\gamma^{-1})'(z)\right]^s \rho(\gamma) f(\gamma^{-1}(z)).
\end{equation}
These operators are well-defined provided $ \gamma^{-1}( \mathbb{D} ) \subset \mathbb{D}. $ The following result shows that $ \mathcal{L}_{s,w,\rho} $ is a trace-class operator.

\begin{prop}\label{prop:preliminary_trace_bound}
Let assumptions and notations be as above. Then for every $s\in \mathbb{C}$ there exists a constant $C = C(s,w) > 0$ such that for every integer $n\neq 0$ 
$$
\Vert \nu_{s,\rho}(\gamma_n^{-1}) \Vert_1 \leq  C \frac{\dim( \rho ) }{\vert n\vert^{2\sigma}},
$$
where $\sigma \sceq \mathrm{Re}(s)$. In particular, \eqref{definition_of_TO} defines a bounded trace-class operator 
\begin{equation}\label{Same}
\mathcal{L}_{s, w, \rho}\colon H^2(\mathbb{D} ; V) \to H^2( \mathbb{D} ; V),
\end{equation} 
provided $\sigma > \frac{1}{2}.$ 
\end{prop}

\begin{proof}
The family of functions $ \{ \psi_m\}_{m\in \mathbb{N}_0} $ given by 
\begin{equation}\label{ONB_classical}
\psi_m(z) = \sqrt{\frac{m+1}{\pi}} z^m
\end{equation}
provides an orthonormal basis for the (classical) Bergman space $ H^2(\mathbb{D}) $. Let $ \textbf{e}_1, \dots, \textbf{e}_d $ be a orthonormal basis for the representation space $ V $, where $ d = \dim(V). $ Then the family of functions
\begin{equation}\label{ONB_complete}
\psi_m\cdot \textbf{e}_j
\end{equation}
with $m \in \mathbb{N}_0$ and $1\leq j \leq d$ forms a basis for $ H^2(\mathbb{D}; V) $. Using the singular value estimate in \eqref{min_max_consequence}, we can estimate the singular values of $ \nu_{s,\rho}(\gamma_n^{-1})$ as
\begin{equation}\label{firstBound}
\mu_m( \nu_{s,\rho}(\gamma_n^{-1}) ) \leq  \sum_{k=1}^d \sum_{j \geq m} \Vert \nu_{s,\rho}(\gamma_n^{-1}) \psi_j \textbf{e}_k \Vert.
\end{equation}
Notice that since $\rho$ is a unitary representation, the operator norm of the endomorphism $\rho(\gamma_n)$ satisfies
\begin{equation}\label{cani}
\Vert \rho(\gamma_n) \Vert_{\mathrm{End}(V)} = 1.
\end{equation}
Hence,
\begin{equation}\label{dani}
\Vert \nu_{s,\rho}(\gamma_n^{-1}) \psi_j \textbf{e}_k \Vert^2 =  \int_{\mathbb{D}} \left\vert \gamma_n'(z)^s \psi_j(\gamma_n(z)) \right\vert^2 \dvol(z) = \frac{j+1}{\pi} \int_{\mathbb{D}} \left\vert \gamma_n'(z)^s \gamma_n(z)^j\right\vert^2 \dvol(z)
\end{equation}
The goal now is to estimate the integral on the right hand side of \eqref{dani}. Observe that for every nonzero integer $n$ the M\"obius transformation $\gamma_n$ maps the open unit disk to 
$$ \left\{ z\in \mathbb{C} : \left\vert z - \frac{1}{nw} \right\vert < \frac{1}{\vert n\vert w} \right\}. $$ 
It follows that the image $\gamma_n(\mathbb{D})$ is contained in the disk
$$
D_{\mathbb{C}}(0, r) = \{ z\in \mathbb{C} : \vert z \vert < r \}
$$
with radius $r=\frac{2}{w} < 1$ and hence, for all $z\in \mathbb{D}$ we have
\begin{equation}\label{ani}
\left\vert \gamma_n(z)\right\vert \leq r.
\end{equation}
In what follows the implied constants depend only on $s$ and $w$. In light of \eqref{complex_power} we have for all $ n\neq 0 $ and $ z\in \mathbb{D} $ the bound
\begin{equation}\label{bani}
\left\vert \gamma_n'(z)^s \right\vert \ll \frac{1}{\vert n\vert^{2\sigma}}.
\end{equation}
Combining \eqref{ani} and \eqref{bani} we get
$$
\Vert \nu_{s,\rho}(\gamma_n^{-1}) \psi_j \textbf{e}_k \Vert^2 \ll \frac{j+1}{\vert n\vert^{4\sigma}}  \int_{\mathbb{D}} \left\vert \gamma_n(z) \right\vert^{2j} \dvol(z) \ll \frac{j+1}{\vert n\vert^{4\sigma}} r^{2j} 
$$
Thus, going back to \eqref{firstBound} and recalling that $0 < r < 1$, we estimate
$$
\mu_m( \nu_{s,\rho}(\gamma_n^{-1}) )  \ll \frac{\dim( \rho ) }{\vert n\vert^{2\sigma}} \sum_{j=m}^\infty \sqrt{j+1} \, r^{j} \ll \frac{\dim( \rho ) }{\vert n\vert^{2\sigma}} \sqrt{m+1} \, r^m.
$$
We deduce that
$$
\Vert \nu_{s,\rho}(\gamma_n^{-1})\Vert_1 = \sum_{m=1}^\infty \mu_m( \nu_{s,\rho}(\gamma_n^{-1}) ) \ll \frac{\dim( \rho ) }{\vert n\vert^{2\sigma}}.
$$
Moreover, since the trace norms $\Vert \nu_{s,\rho}(\gamma_n^{-1})\Vert_1 $ for all integers $n\neq 0$ are summable in $n$ provided $\sigma > \frac{1}{2}$, the operator $\mathcal{L}_{s,w,\rho}$ is trace-class. This completes the proof.
\end{proof}

\subsection{Proof of Theorem \ref{TO_main_theorem}}
Combining Proposition \ref{prop:preliminary_trace_bound} with the fact that $ \mathcal{L}_{s, w,\rho} $ depends holomorphically on $ s $, we deduce that the Fredholm determinant 
$$ 
\det\left( 1- \mathcal{L}_{s, w,\rho} \right) 
$$ 
is a holomorphic function in the half-plane $ \mathrm{Re}(s) > \frac{1}{2}. $ Our goal is to show that it coincides with the twisted Selberg zeta function $ Z_{\Gamma_{ w}}(s,\rho) $ in this half-plane. To that effect fix $ s\in \mathbb{C} $ with $ \mathrm{Re}(s) > \frac{1}{2} $ and consider the entire function $ u\mapsto \det\left( 1-u \mathcal{L}_{s, w, \rho} \right) $. Recall from \eqref{Fredholm_expansion_0} that we have the absolutely convergent expansion
\begin{equation}\label{Fredholm_expansion}
\det\left( 1- u \mathcal{L}_{s,w, \rho} \right) =  \exp\left( - \sum_{N=1}^{\infty} \frac{u^{N}}{N} \mathrm{tr}\left( \mathcal{L}_{s, w,\rho}^{N} \right) \right),
\end{equation}
provided $ \vert u\vert $ is small enough. In view of \eqref{Fredholm_expansion}, Theorem \ref{TO_main_theorem} amounts to finding a suitable expression for the traces of the iterates $ \mathcal{L}_{s, w,\rho}^N $. To do so, notice that the operators in \eqref{compostion_operators} satisfy the composition rule  
$$ \nu_{s,\rho}(g_1)\nu_{s,\rho}(g_2) = \nu_{s,\rho}(g_1 g_2), $$
which in turn implies that
$$
\mathcal{L}_{s, w,\rho}^{N} = \left( \sum_{n\in \ZZ\smallsetminus\{0\}} \nu_{s,\rho}(T_{w}^{n} S) \right)^N = \sum_{n_1,\ldots, n_N\in\ZZ\smallsetminus\{0\}} \nu_{s,\rho}(T_{w}^{n_{1}}S T_{w}^{n_{2}}S \cdots T_{w}^{n_{N}}S)
$$
for every positive integer $ N $. We can rewrite this more conveniently as 
\begin{equation}\label{N_th_power_Hecke}
\mathcal{L}_{s, w,\rho}^{N} = \sum_{\gamma \in P_{N}} \nu_{s,\rho}(\gamma),
\end{equation}
where $P_{N}\subset \Gamma_w$ is the set
\[
 P_N \sceq \left\{ T_{w}^{n_{1}}S T_{w}^{n_{2}}S \cdots T_{w}^{n_{N}}S : n_1,\ldots, n_N\in\ZZ\smallsetminus\{0\}\right\}.
\]
Now let $P \subset \Gamma_w$ be the union of all the $P_N$'s,
$$
P \sceq \bigcup_{N\in \mathbb{N}} P_{N}.
$$
Recall that $ [\Gamma_{w}]_{h} $ and $ [\Gamma_{w}]_p $ denote the set of conjugacy classes of hyperbolic elements in $ \Gamma_w $ and the set of conjugacy classes of primitive hyperbolic elements in $ \Gamma_w $ respectively. Given a conjugacy class $ [\gamma] $ represented by a hyperbolic element $ \gamma\in \Gamma_{w} $ we denote by $ m(\gamma) $ the unique positive integer $ m $ satisfying $ \gamma = \widehat{\gamma}^{m} $ with $ \widehat{\gamma}\in \Gamma $ primitive (i.e. $ [\widehat{\gamma}]\in [\Gamma_{w}]_{p} $).

The following properties can be checked easily:
\begin{itemize}
\item[(1)] Every element in $ P $ is hyperbolic. This is true because the only non-hyperbolic elements are those which are conjugated to powers of either $ S $ or $ T_w $ (since $ w > 2 $), none of which appear in the set $ P. $

\item[(2)] Every hyperbolic conjugacy class $ [\gamma]\in [\Gamma_{w}]_{h} $ has a representative in $ P $, say in $ P_{N} $, and $ N = N(\gamma) $ is unique with this property.  

\item[(3)] Every hyperbolic conjugacy class $ [\gamma]\in [\Gamma_{ w}]_{h} $ has precisely $ N(\gamma)/m(\gamma) $ distinct representatives in $ P_{N(\gamma)} $. Indeed, after conjugation we can represent $ \gamma $ by a word of the form $ \gamma_{n_1} \cdots \gamma_{n_{N(\gamma)}} $ where $ \gamma_{n_i} = S T_{w}^{n_{i}} $. This word has precisely $ N(\gamma)/m(\gamma) $ distinct cyclic permutations.
\end{itemize}

The following lemma is crucial to make the connection between Selberg zeta functions and transfer operators.

\begin{lemma}\label{lemma_trace}
For every hyperbolic M\"obius transformation $ \gamma\in \Gamma_{w} $ with $ \overline{\gamma^{-1}(\mathbb{D})} \subset \mathbb{D} $ we have
\begin{equation}\label{trace}
 \mathrm{tr}(\nu_{s,\rho}(\gamma)) = \chi(\gamma) \frac{e^{-s \ell(\gamma)}}{1-e^{-\ell(\gamma)}},
\end{equation}
where $ \chi = \tr_V \rho $ is the character associated to the representation $ \rho $ and $ \ell(\gamma) $ is the displacement length of $ \gamma $ given by \eqref{displacement}.
\end{lemma}

Results similar to Lemma \ref{lemma_trace} are widely known in the literature, at least for the trivial representation $ \rho = \textbf{1} $, in which  case it can be seen as a special case of the holomorphic Lefschetz fixed point formula, see for instance Lemma 15.9 in \cite{Borthwick_book} and the references given therein. We will give a proof of Lemma \ref{lemma_trace} at the end of this section for the sake of keeping the proof of Theorem \ref{TO_main_theorem} self-contained.

Taking traces on both sides of \eqref{N_th_power_Hecke}, using Lemma \ref{lemma_trace} and a geometric series expansion, we obtain
$$
\mathrm{tr}\left(\mathcal{L}_{s, w,\rho}^{N}\right) =  \sum_{\gamma \in P_{N}} \chi(\gamma) \frac{e^{-s \ell(\gamma)}}{1-e^{-\ell(\gamma)}} = \sum_{k=0}^{\infty} \sum_{\gamma \in P_{N}} \chi(\gamma) e^{-(s+k)\ell(\gamma)}.
$$
Using the properties (1), (2), and (3) above, we can rewrite the inner sum on the right as a sum over primitive hyperbolic conjugacy classes:
\begin{align*}
\sum_{\gamma \in P_{N}} e^{-(s+k)\ell(\gamma)} \chi(\gamma) &= \sum_{m=1}^{\infty}\sum_{\substack{ \gamma \in P_{N} \\ \gamma = \widehat{\gamma}^{m}, \; [\widehat{\gamma}]\in [\Gamma_{ w}]_{p}}}  \chi(\widehat{\gamma}^{m}) e^{-m(s+k)\ell(\widehat{\gamma})}\\
&= \sum_{m=1}^{\infty} \sum_{\substack{[\widehat{\gamma}]\in [\Gamma_{ w}]_{p}\\ N(\widehat{\gamma})\cdot m = N}} \frac{N}{m}  \chi(\widehat{\gamma}^{m}) e^{-m(s+k)\ell(\widehat{\gamma})}.
\end{align*}
Hence, going back to \eqref{Fredholm_expansion}, we obtain
\begin{align*}
\log \det\left( 1- u \mathcal{L}_{s, w,\rho} \right) &= - \sum_{N=1}^{\infty} \frac{u^{N}}{N}\mathrm{tr}\left(\mathcal{L}_{s, w,\rho}^{N}\right)\\
&= - \sum_{N=1}^{\infty} \frac{u^{N}}{N} \sum_{k=0}^{\infty}\sum_{m=1}^{\infty} \sum_{\substack{[\widehat{\gamma}]\in [\Gamma_{ w}]_{p}\\ N(\widehat{\gamma})\cdot m = N}} \frac{N}{m}  \chi(\widehat{\gamma}^{m}) e^{-(s+k)\ell(\gamma)}.
\end{align*}
Rearranging the order of summation (which is justified for $ \mathrm{Re}(s) $ large enough by absolute convergence) leads to
\begin{align*}
\log \det\left( 1- u \mathcal{L}_{s, w,\rho} \right) &= - \sum_{k=0}^{\infty} \sum_{[\widehat{\gamma}]\in [\Gamma_{ w}]_{p}} \sum_{m=1}^{\infty} \frac{u^{N(\widehat{\gamma})\cdot m}}{m}  \chi(\widehat{\gamma}^{m}) e^{-m(s+k)\ell(\widehat{\gamma})}\\
&= \sum_{k=0}^{\infty} \sum_{[\widehat{\gamma}]\in [\Gamma_{ w}]_{p}}  \log \det\left( 1- u^{N(\widehat{\gamma})} \rho(\widehat{\gamma}) e^{-(s+k)\ell(\widehat{\gamma})}\right)\\
&= \log \prod_{k=0}^{\infty} \prod_{[\widehat{\gamma}]\in [\Gamma_{w}]_{p}} \det \left( 1- u^{N(\widehat{\gamma})} \rho(\widehat{\gamma}) e^{-(s+k)\ell(\widehat{\gamma})}\right). 
\end{align*}
Recall from the discussion at the end of Subsection \ref{standard_facts} that the expression in the last line converges at $ u=1 $, provided $ \mathrm{Re}(s) $ is large enough. Thus we obtain the identity 
\begin{equation}\label{reference_to_identity}
Z_{\Gamma_{w}}(s,\rho) = \det\left( 1- \mathcal{L}_{s, w, \rho} \right),
\end{equation}
completing the proof of Theorem \ref{TO_main_theorem}, provided $ \mathrm{Re}(s) $ is large enough. Since both sides of \eqref{reference_to_identity} are holomorphic functions in the half-plane $\mathrm{Re}(s) > \frac{1}{2}$, the validity of this identity extends to $ \mathrm{Re}(s) > \frac{1}{2}, $ by uniqueness of analytic continuation.

\begin{proof}[Proof of Lemma \ref{lemma_trace}]
Let $ \{ \psi_m\}_{m\geq 0} $ be the orthonormal basis of the space $H^2(\mathbb{D})$ given by \eqref{ONB_classical}. We can then explicitly compute the associated Bergman kernel
\begin{equation}\label{Bergman_kernel}
B_{\mathbb{D}}(z,z')=\sum_{m=0}^\infty \psi_m(z)\overline{\psi_m(z')} = \frac{1}{\pi (1-z \overline{z'})^2}.
\end{equation}
Recall also that after having fixed a basis $ \textbf{e}_1, \dots, \textbf{e}_d $ for the representation space $V$, the family 
$$ \{ \psi_m\cdot \textbf{e}_j \}_{\substack{ m \geq 0\\  1\leq j\leq d}} $$ 
provides an orthonormal basis for  $H^2(\mathbb{D}, V)$. Using this basis, we compute the trace as
\begin{align*}
 \mathrm{tr}(\nu_{s,\rho}(\gamma)) &= \sum_{m=0}^{\infty} \sum_{j=1}^{d} \left\langle \nu_{s,\rho}(\gamma)(\psi_m\cdot \textbf{e}_j), \psi_m\cdot \textbf{e}_j  \right\rangle\\
&= \sum_{m=0}^{\infty} \sum_{j=1}^{d} \int_{\mathbb{D}} \langle \rho(\gamma) \textbf{e}_j, \textbf{e}_j \rangle_V  \left[ (\gamma^{-1})'(z)\right]^s \psi_m(\gamma^{-1}(z)) \overline{\psi_m(z)} \dvol(z)\\
&= \left( \sum_{j=1}^{d}  \langle \rho(\gamma) \textbf{e}_j, \textbf{e}_j \rangle_V \right) \int_{\mathbb{D}} \left[ (\gamma^{-1})'(z)\right]^s B_{\mathbb{D}}(\gamma^{-1}(z),z) \dvol(z).
\end{align*}
The parenthetical sum in the previous line is equal to $ \tr_V(\rho(\gamma)) = \chi(\gamma) $, so it remains to calculate the integral. Using the explicit formula for the Bergman kernel in \eqref{Bergman_kernel} we can write
$$
\int_{\mathbb{D}} \left[ (\gamma^{-1})'(z)\right]^s B_{\mathbb{D}}(\gamma^{-1}(z),z) \dvol(z) = \frac{1}{\pi}\int_{\mathbb{D}} \frac{\left[ (\gamma^{-1})'(z)\right]^s}{\left( 1 - \overline{z} \gamma^{-1}(z)  \right)^2} \dvol(z).
$$
Now we apply the complex form of Stokes' formula
$$
\int_{\mathbb{D}} \frac{\partial F}{\partial \overline{z}} \dvol(z) = \frac{1}{2i}\int_{\vert z\vert = 1} F dz,
$$
valid for any $ F\in C^1(\overline{\mathbb{D}}) $, to the function 
$$
F(z,\overline{z}) = \frac{ \overline{z} \left[ (\gamma^{-1})'(z)\right]^s}{1- \overline{z} \gamma^{-1}(z)}.
$$
This yields
\begin{equation}\label{to_use_Cauchy}
\int_{\mathbb{D}} \left[ (\gamma^{-1})'(z)\right]^s B_{\mathbb{D}}(\gamma^{-1}(z),z) \dvol(z) = \frac{1}{2 \pi i} \int_{\vert z\vert = 1} \frac{ \overline{z} \left[ (\gamma^{-1})'(z)\right]^s}{1- \overline{z} \gamma^{-1}(z)} dz =  \frac{1}{2 \pi i} \int_{\vert z\vert = 1} \frac{  \left[ (\gamma^{-1})'(z)\right]^s}{z -  \gamma^{-1}(z)} dz,
\end{equation}
where in the last equation we used the fact that the integration on the left is restricted to $\vert z\vert^2 = z \overline{z} = 1 $. Using the Cauchy integral formula, the integral on the right hand side of \eqref{to_use_Cauchy} can be evaluated to be equal to 
$$
\frac{  \left[ (\gamma^{-1})'(z_0)\right]^s}{1 -(\gamma^{-1})'(z_0)},
$$
where $ z_0 $ is the (unique) fixed point of the map $ \gamma^{-1}\colon \mathbb{D} \to \mathbb{D}. $ Finally, one can show by an elementary calculation that $ (\gamma^{-1})'(z_0) = e^{-\ell(\gamma)} $, completing the proof of Lemma \ref{lemma_trace}.
\end{proof}

\subsection{Proof of Corollary \ref{main_corollary}}
The goal of this subsection is to prove Corollary \ref{main_corollary}. In view of Theorem \ref{TO_main_theorem} it suffices to show that $ s\mapsto \mathcal{L}_{s, w,\rho} $ (which is only defined for $ \mathrm{Re}(s) > \frac{1}{2} $) admits a meromorphic continuation to $ s\in \mathbb{C} $ with poles contained in $ \frac{1}{2}(1-\mathbb{N}_0) $. We will use ideas of Mayer \cite{Mayer-meromorphic} and Pohl \cite{Pohl_representation}. 

Every $ f\in H^{2}(\mathbb{D}; V) $, being holomorphic, can be Taylor-expanded around $z=0$ as
$$
f(z) = \sum_{m=0}^\infty c_m z^m
$$
for some suitable coefficients $c_m \in V$. Hence, we can write
$$
f(z) = f(0) + z \widetilde{f}(z), \quad z\in \mathbb{D},
$$
where $\widetilde{f} \in H^{2}(\mathbb{D}; V)$ is given by
$$
\widetilde{f}(z) = \sum_{m=0}^\infty c_{m+1} z^m
$$
We can then write for all $ z\in \mathbb{D} $
\begin{align*}
\mathcal{L}_{s, w, \rho} f(z) &= \sum_{n\in \ZZ \smallsetminus \{ 0\}} \gamma_n'(z)^{s} \rho(\gamma_n)^{-1} f\left(\gamma_n(z) \right)\\
&= \sum_{n\in \ZZ \smallsetminus \{ 0\}} \gamma_n'(z)^{s} \rho(\gamma_n)^{-1} f(0) +  \sum_{n\in \ZZ \smallsetminus \{ 0\}} \gamma_n'(z)^{s} \gamma_{n}(z) \rho(\gamma_n)^{-1} \widetilde{f}\left( \gamma_{n}(z) \right)
\end{align*}
Using the relation
$$
\gamma_n'(z)^{s} \gamma_{n}(z) = - \gamma_n'(z)^{s+1/2},
$$
this gives 
\begin{equation}\label{massage}
\mathcal{L}_{s, w, \rho} f(z) = \left( \sum_{n\in \ZZ \smallsetminus \{ 0\}} \gamma_n'(z)^{s} \rho(\gamma_n)^{-1} \right) f\left(0 \right) - \underbrace{\sum_{n\in \ZZ \smallsetminus \{ 0\}} \gamma_n'(z)^{s+1/2} \rho(\gamma_n)^{-1} \widetilde{f} \left(\gamma_n(z) \right)}_{ = \mathcal{L}_{s+1/2, w, \rho} \widetilde{f}(z)}.
\end{equation}
Now let us introduce the (bounded) operator
\begin{equation}\label{Psi}
\Psi \colon  H^{2}(\mathbb{D}; V) \to H^{2}(\mathbb{D}; V), \quad  f \mapsto -\widetilde{f}
\end{equation}
and the (finite-rank) operator 
\begin{equation}\label{finite_rank}
\mathcal{F}_{s,w,\rho,1} \colon H^{2}(\mathbb{D}; V) \to H^{2}(\mathbb{D}; V), \quad f \mapsto \xi(s,w,\rho; \cdot) f\left(0 \right),
\end{equation}
where
\begin{equation}\label{almost_hurwitz_stuff}
\xi(s,w,\rho; z) \sceq \sum_{n\in \ZZ \smallsetminus \{ 0\}} \gamma_n'(z)^{s} \rho(\gamma_n)^{-1}.
\end{equation}
We can then rewrite \eqref{massage} more conveniently as 
\begin{equation}\label{convenient}
\mathcal{L}_{s, w,\rho} = \mathcal{F}_{s,w,\rho,1} + \mathcal{L}_{s+\frac{1}{2}, w,\rho} \Psi.
\end{equation}
The second term on the right hand side of \eqref{convenient} is obviously defined for all $ \mathrm{Re}(s)>0 $, while the first term is defined a priori only in the range $ \mathrm{Re}(s) > \frac{1}{2}. $ To pass beyond $ \mathrm{Re}(s) = \frac{1}{2} $ we have to study the operator in \eqref{finite_rank}. Recalling that $ \gamma_n = S T_w^n $, we can rewrite \eqref{almost_hurwitz_stuff} as
\begin{equation}\label{splitting_S}
\xi(s,w,\rho; z) = \left( \sum_{n\in \ZZ \smallsetminus \{ 0\}} \gamma_n'(z)^{s} \rho(T_w^{-1})^{n} \right)\rho(S).
\end{equation}
Furthermore, since $ \rho $ is unitary, we can find real numbers $ \mu_1, \dots, \mu_d \in [0, 1) $ and a  basis $ \textbf{e}_1, \dots, \textbf{e}_d $ of the representation space $ V $, with respect to which $ \rho(T_w^{-1}) $ acts by the diagonal matrix
\begin{equation}\label{diagonalization}
\rho(T_w^{-1}) = \diag\left( e^{2\pi i \mu_1}, \cdots, e^{2\pi i \mu_d} \right).
\end{equation}
Inserting \eqref{diagonalization} into \eqref{splitting_S}, we obtain the expression
\begin{equation}\label{diagonal_lerch}
\xi(s,w,\rho; z) = \diag\left( \sum_{n\in \ZZ \smallsetminus \{ 0\}} \gamma_n'(z)^{s} e^{2\pi i \mu_1 n}, \cdots, \sum_{n\in \ZZ \smallsetminus \{ 0\}} \gamma_n'(z)^{s} e^{2\pi i \mu_d n} \right) \cdot \rho(S).
\end{equation}
Let us now inspect the diagonal entries on the right of \eqref{diagonal_lerch} individually. Recalling the definition of the complex powers $ \gamma_n'(z)^{s} $ given in \eqref{complex_power}, we can write for each $ j\in \{ 1, \dots, d\} $:
\begin{align*}
\sum_{n\in \ZZ \smallsetminus \{ 0\}} \gamma_n'(z)^{s} e^{2\pi i \mu_j n} &= \sum_{n=1}^\infty \frac{e^{2\pi i \mu_j n}}{(nw+z)^{2s}}  + \sum_{n=1}^\infty \frac{ e^{-2\pi i \mu_j n}}{(nw-z)^{2s}}\\
&= w^{-2s} H\left( \frac{z}{w}, 2s, \mu_j \right) + w^{-2s} H\left( -\frac{z}{w}, 2s, -\mu_j \right),
\end{align*}
where
\begin{equation}\label{Lerch}
H(z,s,\mu) \sceq \sum_{n=1}^\infty \frac{e^{2\pi i \mu n}}{(n+z)^{s}} 
\end{equation}
is the \textit{Lerch zeta function}. The analytical properties of $ H(z,s,\mu)$ are well-known in the literature, see for instance \cite{Apostol}. Given $ \mu\in [0,1) $ and $ 0 < r < 1 $, the Lerch zeta function defines a holomorphic map
$$
\{ \vert z\vert < r\} \times \left( \CC \smallsetminus \mathcal{P} \right) \to \mathbb{C}, \quad (z,s) \mapsto H(z,s,\mu),
$$
where $ \mathcal{P} = 1-\mathbb{N}_0 $ is the set of (potential) poles. Consequently,
$$
\mathbb{D} \times \left( \CC \smallsetminus \frac{1}{2}(1-\mathbb{N}_0) \right) \to \mathrm{End}(V), \quad (z,s) \mapsto \xi(s,w,\rho; z)
$$
is a holomorphic map with values in the endomorphism ring of $V$. This in turn shows that
$$
s \mapsto \left( \mathcal{F}_{s,w,\rho,1} \colon H^{2}(\mathbb{D}; V)\to H^{2}(\mathbb{D}; V) \right)
$$
is a family of operators depending meromorphically on $ s $ with poles contained in $ \frac{1}{2}\left( 1-\mathbb{N}_0 \right). $ Going back to \eqref{convenient} we have thus shown that $ \mathcal{L}_{s, w,\rho} $ admits a meromorphic continuation to the half-plane $ \mathrm{Re}(s) > 0 . $ To extend $ \mathcal{L}_{s, w,\rho} $ further to the left, we take an arbitrary positive integer $ k\in \mathbb{N} $ and iterate equation \eqref{convenient} $ k $ times, where in each iteration step the `current' variable $ s $ gets replaced by $ s+\frac{1}{2} $. This procedure yields 
\begin{equation}\label{k-convenient}
\mathcal{L}_{s, w, \rho}  = \mathcal{F}_{s,w,\rho,k} + \mathcal{L}_{s+\frac{k}{2}, w, \rho} \Psi^{k},
\end{equation}
where
\begin{equation}\label{iterated}
\mathcal{F}_{s,w,\rho,k} = \sum_{j=0}^{k-1} \mathcal{F}_{s+\frac{j}{2},w,\rho,1}\Psi^j.
\end{equation}
From the already established analytic properties of $ \mathcal{F}_{s,w,\rho,1}, $ we infer from the right hand side of \eqref{iterated} that
$$
s \mapsto \left( \mathcal{F}_{s,w,\rho,k} \colon H^{2}(\mathbb{D}; V)\to H^{2}(\mathbb{D}; V) \right)
$$
is a meromorphic family of operators with poles in $ \frac{1}{2}\left( 1-\mathbb{N}_0 \right) $ for all $ k\in \mathbb{N} $. Hence, \eqref{k-convenient} shows meromorphic continuability of $ \mathcal{L}_{s, w, \rho} $ on the half-plane $ \mathrm{Re}(s) > \frac{1-k}{2} $ for arbitrary $ k\in \mathbb{N}. $ This settles the proof of Corollary \ref{main_corollary}.

\section{Proof of Theorem \ref{zeros_and_eigenvalues}}\label{Section_2}

The goal of this section is to prove Theorem \ref{zeros_and_eigenvalues}. Recall from the introduction that for every positive integer $n$ we define the subgroup $\Gamma_w^n \subset \Gamma_w$ as the kernel 
$$ \Gamma_w^n = \ker( \rho_n ) $$ 
of the one-dimensional character $ \rho_n\colon \Gamma_w \to\mathbb{C}^\times$ given by
$$
\rho_n (S) = -1 \quad \text{and} \quad \rho_n(T_w) = e^{\frac{2\pi i}{n}}.
$$ 
The groups $\Gamma_w^n$ are finite-index, normal subgroups of $\Gamma_w$. One important feature of these groups is that we can provide a complete set of representatives in $\Gamma_w$ of the left cosets in $\Gamma_w/\Gamma_w^n$, namely
$$
\Gamma_w/\Gamma_w^n \simeq \{ T_w^a, T_w^a S : 0\leq a \leq n-1 \}.
$$
Note that this set forms an abelian group, namely
$$
\Gamma_w/\Gamma_w^n \simeq \mathbb{Z}/n\mathbb{Z} \times  \mathbb{Z}/2\mathbb{Z} \simeq \mathbb{Z}/2n\mathbb{Z}.
$$
Similarly, the groups $\Gamma_w^n$ are finite-index, normal subgroup of $\Gamma_w^1$ with a complete set of representatives in $\Gamma_w^1$ of the left cosets in $\Gamma_w^1/\Gamma_w^n$ given by 
\begin{equation}\label{cosets_for_Gamma_1_in_Gamma_n}
\Gamma_w^1/\Gamma_w^n \simeq \{ T_w^a : 0\leq a \leq n-1 \} \simeq \mathbb{Z}/n\mathbb{Z}.
\end{equation}
Using \eqref{cosets_for_Gamma_1_in_Gamma_n} we can construct a fundamental domain $ \mathcal{F}^n(w) $ for the action of $\Gamma_w^n$ on $\mathbb{H}^2$ as
\begin{equation}\label{fund_domain_n}
\mathcal{F}^n(w) = \bigcup_{a = 0}^{n-1} T_w^a.\mathcal{F}^1(w),
\end{equation}
where $\mathcal{F}^1(w)$ is the fundamental domain for $\Gamma^1_w$ given in \eqref{fundamental_domain_1}. In particular, $\mathcal{F}^n(w)$ is a finite disjoint union of $n$ translates of $\mathcal{F}^1(w)$ (ignoring the boundaries) and the quotients $X_w^n = \Gamma_w^n\backslash \mathbb{H}^2$ are covers of $X_w^1$ of degree $n$.

Notice that $ X_w^1 = \Gamma_w^1\backslash \mathbb{H}^2 $ is a smooth hyperbolic surface (no conical singularities!) with one funnel ($ n_f = 1 $), two cusps ($ n_c = 2 $), and genus zero ($ g = 0 $). This shows that $ X_w^1 $ has Euler characteristic 
$$\chi(X_w^1) = 2-2g-n_c-n_f = -1.$$ 
Consequently, $ X_w^n $ is a smooth hyperbolic surface with Euler characteristic $$\chi(X_w^n) = n \cdot \chi(X_w^1) = -n. $$

Part \eqref{Part_2} of Theorem \ref{zeros_and_eigenvalues} is now a straightforward consequence of two well-known results in the spectral theory of hyperbolic surfaces, which we recall here. First, given an arbitrary torsion-free, finitely generated Fuchsian group $ \Gamma $, the result of Borthwick--Judge--Perry \cite{BJP} asserts that the zeros of the Selberg zeta function $ Z_\Gamma(s) $ in $ \left\{ \mathrm{Re}(s) > \frac{1}{2}  \right\} $ correspond, with multiplicities, to the $ L^{2} $-eigenvalues $ \lambda = s(1-s) \in (0, \frac{1}{4}) $ of the Laplacian $ \Delta_X $ on $ X = \Gamma\backslash\HH^2 $. Second, from Ballmann--Mathiesen--Mondal \cite{BMM} we know that the number of eigenvalues of $ \Delta_X $ in $ (0, \frac{1}{4}) $ is bounded above by $ - \chi(X) $, where $ \chi(X) $ denotes the Euler characteristic of the surface $ X $. 

Let us now prove Part \eqref{Part_1} of Theorem \ref{zeros_and_eigenvalues} which says that $s= \delta(w)$ is the unique zero of $Z_{\Gamma_{w}}(s)$ in the half-plane $\mathrm{Re}(s) > \frac{1}{2}.$ Unfortunately, the result of Borthwick--Judge--Perry mentioned above does not apply directly to the Hecke triangle group $\Gamma_w$, since it is not torsion-free (indeed, $\Gamma_w$ contains the element $S$ which satisfies $S^2 = \id$). We need some additional arguments to bypass this issue. We will write $\delta = \delta(w) $ for the remainder of this section. 

We can use the product definition of the Selberg zeta function $Z_{\Gamma}(s,\rho)$ in \eqref{defi_L_functions_2} to compute its logarithmic derivative in the half-plane $\mathrm{Re}(s) > \delta$ as
\begin{equation}\label{log_deriv_H}
\frac{Z_{\Gamma}'(s)}{Z_{\Gamma}(s)} = \sum_{[\gamma]\in [\Gamma_w]_{p}} \sum_{k=0}^{\infty} \sum_{m=1}^\infty \chi(\gamma^m) \ell(\gamma) e^{-m(s+k)\ell(\gamma)},
\end{equation}
where $ \chi(\gamma) \sceq \tr_V( \rho(\gamma) ) $ is the character of the representation $\rho\colon \Gamma_w \to \mathrm{U}(V)$.

Now we can either invoke the Venkov--Zograf factorization formula (see \cite{Venkov_Zograf, Venkov_book} or \cite[Theorem~6.1]{FP_szf}) or directly prove that the Selberg zeta function of $\Gamma_w^1$ factorizes as
\begin{equation}\label{factorization}
Z_{\Gamma_{w}^{1}}(s) = Z_{\Gamma_{w}}(s)Z_{\Gamma_{w}}(s,\rho_1).
\end{equation}
Applying Part \eqref{Part_2} to the case $n=1$ shows that $ Z_{\Gamma_{w}^{1}}(s) $ has exactly one zero in the half-plane $\mathrm{Re}(s) > \frac{1}{2}$. By Corollary \ref{main_corollary}, both $Z_{\Gamma_{w}}(s)$ and $Z_{\Gamma_{w}}(s,\rho_1)$ are holomorphic in the half-plane $\mathrm{Re}(s) > \frac{1}{2}$. Combining these facts with the factorization in \eqref{factorization} immediately implies that $Z_{\Gamma_{w}}(s)$ has at most one zero in $\mathrm{Re}(s) > \frac{1}{2}$, possibly at $s=\delta$. Let us suppose by contradiction that $ Z_{\Gamma_{w}}(s) $ has no zero at $s=\delta$. In that case $ Z_{\Gamma_{w}}(s,\rho_1) $ has one zero at $\delta$, and $ Z_{\Gamma_{w}}(s) $ has no zeros at all in the half-plane $\mathrm{Re}(s) > \frac{1}{2}.$ Then, using \eqref{log_deriv_H} and standard methods of analytic number theory we can use this information on the zeros of $Z_{\Gamma_{w}}(s, \rho_1)$ and $Z_{\Gamma_{w}}(s)$ to obtain the asymptotics
$$
\sum_{[\gamma]\in [\Gamma_w]_{p}} \rho_1(\gamma) \ell( \gamma ) \sim e^{\delta x}, \quad x\to \infty
$$
and
$$
\sum_{\substack{ [\gamma] \in [\Gamma]_p \\ \ell(\gamma) \leq x }} \ell( \gamma )  = O_\varepsilon(e^{ (\frac{1}{2}+\varepsilon)x }), \quad x \to \infty
$$
for all $\varepsilon > 0$. Note that since $ \rho(\gamma) \in \{ \pm 1 \} $, we can trivially bound
\begin{equation}\label{comparingAsympt}
e^{\delta x} \sim \sum_{\substack{ [\gamma] \in [\Gamma]_p \\ \ell(\gamma) \leq x }} \rho_1(\gamma)\ell( \gamma )  \leq \sum_{ \substack{ [\gamma] \in [\Gamma]_p \\ \ell(\gamma) \leq x }}\ell( \gamma ) =  O_\varepsilon(e^{ (\frac{1}{2}+\varepsilon)x })
\end{equation}
for all $\varepsilon > 0$ as $x\to \infty$. Comparing the exponents on both sides of \eqref{comparingAsympt} forces $ \delta \leq \frac{1}{2} $, a contradiction to \eqref{prelim_results}. Hence, $ s=\delta $ is a zero of $ Z_{\Gamma_w}(s) $ and there are no other zeros in $ \mathrm{Re}(s) > \frac{1}{2} $, as claimed.

It remains to show Part \eqref{Part_3}, the proof of which will occupy the remainder of this section. We will adapt the argument given in \cite{JNS} (where it was used to prove a similar statement for Schottky groups). Let us introduce the family of operators
$$
\mathcal{L}_{s,w}^{(\theta)}\colon H^2(\mathbb{D})\to H^2(\mathbb{D})
$$
defined for all real parameters $\theta \in \mathbb{R}$ by
\begin{equation}\label{character_twisted_operator}
\mathcal{L}_{s,w}^{(\theta)}f(z) \sceq \sum_{n\in \mathbb{Z}\smallsetminus \{ 0\}} e^{2\pi i \theta} \gamma_{n}'(z)^s f(\gamma_n(z)).
\end{equation}
We then have the following result.

\begin{lemma}\label{lemma_factorization}
We have the factorization
\begin{equation}\label{fac_eq}
Z_{\Gamma_w^n}(s) = \prod_{a = 0}^{n-1} \det\left( 1- \mathcal{L}_{s,w}^{(a/n)} \right) \cdot \prod_{a = 0}^{n-1} \det\left( 1+\mathcal{L}_{s,w}^{(a/n)} \right)
\end{equation}
and every factor $\det\left( 1 \pm \mathcal{L}_{s,w}^{(a/n)} \right)$ is holomorphic in the half-plane $\mathrm{Re}(s) > \frac{1}{2}.$ 
\end{lemma}

\begin{proof}
Let $ \xi_n \colon \Gamma_w \to \mathbb{C}^\times $ be the representation given by
$$
\xi_n (S) = 1  \quad \text{and} \quad  \xi_n(T_w) = e^{\frac{2\pi i}{n}}.
$$
and recall that $\rho_1$ is the representation of $\Gamma_w$ given by
$$
\rho_1(S) = -1 \quad \text{and} \quad \rho_1(T_w) = 1.
$$
The set of irreducible representations of the (abelian) symmetry group
$$
\Gamma_w/\Gamma^n_w \simeq \mathbb{Z}/n\mathbb{Z} \times \mathbb{Z}/2\mathbb{Z}
$$
is given by the collection of the $2n$ characters
$$
\{ \rho_1 \xi_n^a, \xi_n^a : 0\leq a \leq n-1 \}.
$$
By the Venkov--Zograf formula (citations as above) we can factorize the Selberg zeta function of the subgroup $\Gamma_w^n \subset \Gamma_w$ into a product of twisted Selberg zeta functions of $\Gamma_w$ as
$$
Z_{\Gamma_w^n}(s) = \prod_{a = 0}^{n-1} Z_{\Gamma_w}(s, \xi_n^a) \cdot \prod_{a = 0}^{n-1} Z_{\Gamma_w}(s, \rho_1 \xi_n^a).
$$
(Note that this is a straightforward generalization of the factorization in \eqref{factorization}.) Applying Theorem \ref{TO_main_theorem} to each of the factors appearing on the right hand side, we obtain
$$
Z_{\Gamma_w^n}(s) = \prod_{a = 0}^{n-1}\det\left( 1- \mathcal{L}_{s,w, \xi_n^a} \right) \cdot \prod_{a = 0}^{n-1} \det\left( 1- \mathcal{L}_{s,w, \rho_1  \xi_n^a} \right).
$$
Now note that we can use the notation introduced in \eqref{character_twisted_operator} to write 
$$ \mathcal{L}_{s,w, \xi_n^a} = \mathcal{L}_{s,w}^{(a/n)} \quad \text{and} \quad \mathcal{L}_{s,w, \rho_1 \xi_n^a} = -  \mathcal{L}_{s,w}^{(a/n)}, $$ completing the proof.
\end{proof}

In light of Lemma \ref{lemma_factorization} the proof of Part \eqref{Part_3} can be explained as follows. For $\theta$ close to zero the operator $ \mathcal{L}_{s,w}^{(\theta)} $ is ``close'' to $\mathcal{L}_{s,w} = \mathcal{L}_{s,w}^{(0)}$. In particular, if $\theta \approx 0$ then the Fredholm determinant $ \det\left( 1- \mathcal{L}_{s,w}^{(\theta)}\right) $ must have a zero close to $s = \delta$, which is a zero of $\det\left( 1- \mathcal{L}_{s,w}\right) = Z_{\Gamma_w}(s)$. This in turn implies that each factor $ \det\left( 1- \mathcal{L}_{s,w}^{a/n}\right) $ appearing on the right hand side of \eqref{fac_eq} will produce a zero arbitrarily close to $s=\delta$, provided $a/n$ is sufficiently small. To materialize this idea we need the following result.

\begin{lemma}\label{lemma_for_theta_operators}
For all $\sigma \sceq \mathrm{Re}(s) > \frac{1}{2}$ there exists a constant $C>0$ such that
$$
\Vert \mathcal{L}_{s,w} - \mathcal{L}_{s,w}^{(\theta)} \Vert_1 \leq C  \theta^{2\sigma - 1} 
$$
for all $\theta\in [0,1)$.
\end{lemma}

\begin{proof}
If $\theta = 0$ there is nothing to prove, so we assume that $\theta > 0.$ Notice that we can write
$$
\mathcal{L}_{s,w}^{(\theta)} = \sum_{n\in \mathbb{Z}\smallsetminus \{ 0 \}} e^{ 2 \pi i \theta n } \nu_s(\gamma_n^{-1})
$$
where $  \nu_s(\gamma_n^{-1}) = \nu_{s,\textbf{1}}(\gamma_n^{-1}) $ is the operator given by \eqref{compostion_operators}. Thus
$$
\mathcal{L}_{s,w} - \mathcal{L}_{s,w}^{(\theta)} = \sum_{n\in \mathbb{Z}\smallsetminus \{ 0 \}} \left( 1- e^{ 2 \pi i \theta n } \right) \nu_s(\gamma_n^{-1}) = - \sum_{n\in \mathbb{Z}\smallsetminus \{ 0 \}} e^{ \pi i\theta n } \sin(\theta n) \nu_s(\gamma_n^{-1}) .
$$
Hence, the triangle inequality gives
$$
\Vert \mathcal{L}_{s,w} - \mathcal{L}_{s,w}^{(\theta)} \Vert_1 \leq  \sum_{n\in \mathbb{Z}\smallsetminus \{ 0 \}} \vert \sin(\pi \theta n)\vert \Vert \nu_s(\gamma_n^{-1}) \Vert_1.
$$
Using the bound for the trace norm for $ \nu_s(\gamma_n^{-1})$ in Proposition \ref{prop:preliminary_trace_bound}, we obtain
$$
\Vert \mathcal{L}_{s,w} - \mathcal{L}_{s,w}^{(\theta)} \Vert_1 \leq C  \sum_{n = 1}^\infty \frac{\vert \sin(\pi \theta n)\vert }{\vert n\vert^{2\sigma}}
$$
for some $C > 0$ depending solely on $s$ and $w$. In order to estimate the remaining sum we split it as
$$
\sum_{n = 1}^\infty \frac{\vert \sin( \pi \theta n)\vert }{\vert n\vert^{2\sigma}} = \sum_{1 \leq n < 1/(3\theta)} \frac{\vert \sin( \pi \theta n)\vert }{\vert n\vert^{2\sigma}} + \sum_{n \geq 1/(3\theta)} \frac{\vert \sin( \pi \theta n)\vert }{\vert n\vert^{2\sigma}}
$$
The second sum can be estimated as
\begin{equation}\label{first_bound}
 \sum_{n \geq 1/(3\theta)} \frac{\vert \sin( \pi \theta n)\vert }{\vert n\vert^{2\sigma}} \leq \sum_{n \geq 1/(3\theta)} \frac{1}{\vert n\vert^{2\sigma}} \ll \int_{1/(3\theta)}^\infty \frac{dx}{x^{2\sigma}} \ll \theta^{2\sigma - 1}.
\end{equation}
Using the elementary bound $ \vert \sin(\pi x) \vert <  2 \vert x\vert $ for all $\vert x\vert < 1/2$, we can estimate the first sum as
\begin{equation}\label{second_bound}
\sum_{1 \leq n < 1/(3\theta)} \frac{\vert \sin( \pi \theta n)\vert }{\vert n\vert^{2\sigma}} < \sum_{1 \leq n < 1/(3\theta)} \frac{ n \theta }{\vert n\vert^{2\sigma}} \ll  \theta^{2\sigma -1}.
\end{equation}
Combining \eqref{first_bound} and \eqref{second_bound} completes the proof of Lemma \ref{lemma_for_theta_operators}.
\end{proof}

We are now ready to prove Part \eqref{Part_3} of Theorem \ref{zeros_and_eigenvalues}. In what follows, we assume that $s$ lies in the half-plane $\sigma \sceq \mathrm{Re}(s) > \frac{1}{2}$. From Part \eqref{Part_1} we know that on this half-plane the Selberg zeta function $ Z_{\Gamma_w}(s) $ vanishes \textit{only} at $s=\delta$. Hence, for any fixed $0 < \varepsilon < \delta-\frac{1}{2}$ we have
$$
C(\varepsilon,w) \sceq \inf_{ \vert s- \delta\vert = \varepsilon} \vert Z_{\Gamma_w}(s) \vert > 0.
$$
On the other hand, Theorem \ref{TO_main_theorem} and the estimate in \eqref{Simon_continuity_bound} show that
\begin{align*}
\left\vert Z_{\Gamma_w}(s) - \det\left( 1- \mathcal{L}_{s,w}^{(a/n)} \right) \right\vert &= \left\vert \det\left( 1- \mathcal{L}_{s,w} \right) - \det\left( 1- \mathcal{L}_{s,w}^{(a/n)} \right) \right\vert\\
&\leq \Vert \mathcal{L}_{s,w} - \mathcal{L}_{s,w}^{(a/n)} \Vert_1 \exp\left( \Vert \mathcal{L}_{s,w}^{(a/n)} \Vert_1 + \left\Vert \mathcal{L}_{s,w} \right\Vert_1 + 1  \right).
\end{align*}
Notice that we can use Proposition \ref{prop:preliminary_trace_bound} to show that the trace norms $ \Vert \mathcal{L}_{s,w} \Vert_1 $ and $ \Vert\mathcal{L}_{s,w}^{(a/n)} \Vert_1 $ are bounded from above by a constant depending only on $s$ and $w$. Thus we have
$$
\left\vert Z_{\Gamma_w}(s) - \det\left( 1- \mathcal{L}_{s,w}^{(a/n)} \right) \right\vert \leq C_1 \Vert \mathcal{L}_{s,w} - \mathcal{L}_{s,w}^{(a/n)} \Vert_1
$$
for some constant $C_1 = C_1(s,w)$ not depending on $n$. Applying Lemma \ref{lemma_for_theta_operators} shows furthermore that on the circle $ \vert s-\delta\vert = \varepsilon $ we have
$$
\left\vert Z_{\Gamma_w}(s) - \det\left( 1- \mathcal{L}_{s,w}^{(a/n)} \right) \right\vert \leq C_2 \left( \frac{a}{n} \right)^{2\sigma - 1} \leq C_2 \left( \frac{a}{n} \right)^{2(\delta - \varepsilon) - 1}, 
$$
where $C_2 = C_2(\varepsilon,w)$ is independent of $n$. Thus we can choose a constant $c > 0$ so small that for all integers $n\geq 1$ and $ 0 \leq a \leq c n $ we have 
\begin{equation}\label{rouche}
\left\vert Z_{\Gamma_w}(s) - \det\left( 1- \mathcal{L}_{s,w}^{(a/n)} \right) \right\vert < C(\varepsilon, w).
\end{equation}
Now by Rouch\'{e}'s theorem for holomorphic functions and the fact that $Z_{\Gamma_w}(s)$ vanishes at $s=\delta$, the bound in \eqref{rouche} forces the (holomorphic) function
$$ s \mapsto \det\left( 1- \mathcal{L}_{s,w}^{(a/n)} \right) $$ 
to vanish at some point $s\in \mathbb{C}$ with $ \vert s-\delta\vert < \varepsilon. $ By Lemma \ref{lemma_factorization} this implies that $ Z_{\Gamma_w^n}(s) $ has at least $cn$ zeros (counted with multiplicities) in the disk $ \vert s-\delta\vert < \varepsilon. $ Finally, the result of Borthwick--Judge--Perry \cite{BJP} shows that all these zeros lie on the real interval $(\delta - \varepsilon, \delta]$ and they correspond one-to-one to the eigenvalues $s(1-s)$ of the Laplacian on $X_w^n$. The proof of Theorem \ref{zeros_and_eigenvalues} is now complete.

\section{Hausdorff dimension of Hecke Triangle groups}\label{Section_3}
This section is devoted to the proofs of Theorems \ref{approximation_by_matrices} and \ref{thm:asymptotic_expansion}. We will work solely with the trivial one-dimensional representation $ \rho = \textbf{1} $. Therefore we will drop the representation from the notation of the transfer operator, writing only $ \mathcal{L}_{s,w} $ instead of $\mathcal{L}_{s,\rho, \textbf{1}}$. Recall that $ \mathcal{L}_{s,w} $ acts on the classical Bergman space $H^2(\mathbb{D})$ consisting of \textit{holomorphic} functions on the unit disk with bounded $L^2$-norm. Every function $f\in H^2(\mathbb{D})$ can be Taylor-expanded around $z=0$ as 
\begin{equation}\label{expansion}
f(z) = \sum_{i=0}^{\infty} c_{i}z^{i}
\end{equation}
for some suitable coefficients $ c_i \in \mathbb{C} $. The proofs of Theorems \ref{approximation_by_matrices} and \ref{thm:asymptotic_expansion} are independent but both rely on the following identity.

\begin{prop}\label{main_expression}
For all $\mathrm{Re}(s) > \frac{1}{2}$ and for every $f\in H^2(\mathbb{D})$ with Taylor expansion as in \eqref{expansion} we have the absolutely convergent expression
\begin{equation}\label{main_expression_formala}
\mathcal{L}_{s,w} f(z) = \sum_{i=0}^\infty \sum_{j=0}^\infty  a_{i,j}(s,w) c_j z^i,
\end{equation}
with
$$
a_{i,j}(s,w) = \left( (-1)^{i+j} + 1 \right) \frac{\zeta(2s+i+j)}{w^{2s+i+j}} { 2s+i+j-1 \choose i},
$$
where we use the notation
$$
{ r \choose k} = \frac{r(r-1)\cdots (r-k+1)}{k!}
$$
for all $ r\in \mathbb{C} $ and $k\in \mathbb{N}_0$.
\end{prop}

\begin{proof}
Fix a point $ z\in \mathbb{D} $. By the definition of the transfer operator in \eqref{definition_of_TO}, we write
\begin{align*}
\mathcal{L}_{s,w} f(z) &= \sum_{n\in \mathbb{Z}\smallsetminus \{ 0\}} \gamma_n'(z)^{s} f(\gamma_n(z))\\
&= \sum_{n = 1}^{\infty} \frac{1}{(nw+z)^{2s}} f\left( \frac{-1}{nw+z} \right) + \sum_{n = 1}^{\infty} \frac{1}{(nw-z)^{2s}} f\left( \frac{1}{nw-z} \right).
\end{align*}
Inserting the Taylor expansion for $f$ in the previous line we can rewrite this as
\begin{equation}\label{bracketed_line}
\mathcal{L}_{s,w} f(z) = \sum_{n=1}^{\infty} \sum_{j=0}^\infty c_{j} \left[ \frac{(-1)^j }{(nw+z)^{j+2s}} + \frac{1}{(nw-z)^{j+2s}}   \right].
\end{equation}
We can use the generalized binomial theorem
$$
(1+z)^{r} = \sum_{i=0}^{\infty} { r \choose i} z^{i}, 
$$
valid for all $r\in \mathbb{C}$ and $\vert z\vert< 1$, to rewrite the bracketed expression in \eqref{bracketed_line} as
\begin{align*}
\frac{(-1)^j }{(nw+z)^{j+2s}} + \frac{1}{(nw-z)^{j+2s}} &= \frac{(-1)^j }{(nw)^{j+2s}}\left( 1+ \frac{z}{nw} \right)^{-2s-j} + \frac{1}{(nw)^{j+2s}} \left( 1-\frac{z}{nw}  \right)^{-2s-j}\\
&= \sum_{i=0}^\infty  \left(  (-1)^j + (-1)^i \right)  \frac{1}{(nw)^{2s+i+j}}{ -2s-j \choose i} z^i\\
&= \sum_{i=0}^\infty  \left(  (-1)^{i+j} + 1 \right)  \frac{1}{(nw)^{2s+i+j}}{ 2s+i+j-1 \choose i} z^i,
\end{align*}
where in the last line we used the relation
$$
{ -2s-j \choose i} = (-1)^i { 2s+i+j-1 \choose i}.
$$
Inserting this into the right hand side of \eqref{bracketed_line}, we obtain 
\begin{equation}\label{ultim}
\mathcal{L}_{s,w} f(z) = \sum_{n=1}^{\infty} \sum_{j=0}^\infty  \sum_{i=0}^\infty  \left(  (-1)^{i+j} + 1 \right)  \frac{1}{(nw)^{2s+i+j}}{ 2s+i+j-1 \choose i} c_j z^i.
\end{equation}
Let us now argue why the triple sum in \eqref{ultim} is absolutely convergent. Using orthogonality of the functions $ \{ z^j\}_{j\in \mathbb{N}_0} $ in $ H^2(\mathbb{D}) $, we compute the $L^2$-norm of $ f\in H^2(\mathbb{D}) $ as
\begin{equation}\label{norm}
\Vert f\Vert^{2} = \sum_{j=0}^{\infty} \vert c_j\vert^2 \int_{\mathbb{D}} \vert z\vert^{2j} \dvol(z) = \pi \sum_{j=0}^{\infty}  \frac{\vert c_j\vert^2}{j+1}.
\end{equation}
From \eqref{norm} we deduce that
\begin{equation}\label{trivial_c_k}
\vert c_j\vert \leq \frac{\sqrt{j+1}}{\sqrt{\pi}} \Vert f\Vert
\end{equation}
for all $j\geq 0.$ The bound for binomial coefficients
$$
{r \choose k} \leq {\lceil r\rceil \choose k} \leq 2^{\lceil r\rceil},
$$
valid for all positive reals $r$ and all integers $ 0\leq k\leq r $, yields
\begin{equation}\label{bound_for_binomial_coefficients}
\left\vert   { 2 s +i+j-1 \choose i} \right\vert \leq { 2 \vert s\vert +i+j-1 \choose i} \ll 2^{ 2 \vert s\vert + i + j }.
\end{equation}
Combining \eqref{trivial_c_k} and \eqref{bound_for_binomial_coefficients} shows that the absolute value of each term appearing in the sum \eqref{ultim} is bounded from above by
$$
\left\vert  \left(  (-1)^{i+j} + 1 \right)  \frac{1}{(nw)^{2s+i+j}}{ 2s+i+j-1 \choose i} c_j z^i \right\vert\ll \frac{\sqrt{j+1}}{ n^{2 \mathrm{Re}(s) +i+j}} \left( \frac{2}{w} \right)^{i+j} \Vert f\Vert,
$$
where the implied constant depends solely on the variable $s$. Since $w>2$, this clearly shows that the triple sum on the right hand side of \eqref{ultim} is absolutely convergent for $\mathrm{Re}(s) > \frac{1}{2}.$

Finally, recalling the definition of the Riemann zeta function 
$$ \zeta(s) = \sum_{n=1}^{\infty} \frac{1}{n^s} $$ 
for $ \mathrm{Re}(s) > 1$, we can interchange sums in \eqref{ultim} (allowed by absolute convergence) to write
$$
\mathcal{L}_{s,w} f(z) = \sum_{i=0}^\infty  \sum_{j=0}^\infty  \left(  (-1)^{i+j} + 1 \right)  \frac{\zeta(2s+i+j)}{w^{2s+i+j}}  { 2s+i+j-1 \choose i} c_j z^i,
$$
completing the proof of Proposition \ref{main_expression}.
\end{proof}

\subsection{Proof of Theorem \ref{approximation_by_matrices}}
In this subsection we fix $w > 2$ and we write $\delta = \delta(w).$ Motivated by Proposition \ref{main_expression}, we define for every integer $k>1$ the operator
$$
\mathcal{A}_{s,w,k} \colon H^2(\mathbb{D}) \to H^2(\mathbb{D})
$$
acting on functions $f(z) = c_0 + c_1 z + c_2 z^2+\cdots$ by
\begin{equation}\label{A_k_operator}
\mathcal{A}_{s,w,k} f(z) \sceq  \sum_{i=0}^{k-1} \sum_{j=0}^{k-1} a_{i,j}(s,w) c_j z^i,
\end{equation}
where
$$
a_{i,j}(s,w) = \left( (-1)^{i+j} + 1 \right) \frac{\zeta(2s+i+j)}{w^{2s+i+j}} { 2s+i+j-1 \choose i}.
$$
Let $V_k$ denote the subspace of $H^2(\mathbb{D})$ spanned by the functions $\phi_m(z) = z^m$ with $0\leq m < k$. Notice that the operator $\mathcal{A}_{s,w,k}$ is a finite-rank operator acting by zero on the orthogonal complement of $V_k$. On the subspace $V_k$ the action of $\mathcal{A}_{s,w,k}$ is represented by the $k\times k$-matrix
$$
A_k(s,w) = \left(  a_{i,j}(s,w)\right)_{0\leq i,j < k}
$$
with respect to the basis $\{ \phi_0, \dots, \phi_{k-1}  \}.$ In particular, the determinants of $1 - \mathcal{A}_{s,w,k}$ and $1 - A_k(s,w)$ are identical: 
\begin{equation}\label{equality_determinants}
\det\left( 1-\mathcal{A}_{s,w,k} \right) = \det\left( 1-A_k(s,w) \right) =: D_k(s,w).
\end{equation}
The next result shows that the sequence of operators $\mathcal{A}_{s,w,k}$ converges exponentially fast to $\mathcal{L}_{s,w}$ as $k\to \infty$ with respect to the \textit{trace norm}.

\begin{lemma}\label{lemma:approx_finite_rank}
For all $\mathrm{Re}(s) > \frac{1}{2}$ we have
$$
\Vert \mathcal{L}_{s,w} - \mathcal{A}_{s,w,k} \Vert_1 \leq  C \left( \frac{w}{2} \right)^{- k+o(k)}
$$
where $C = C(s,w) > 0$ is independent of $k$.
\end{lemma}

\begin{proof}
From the formula given in \eqref{main_expression_formala} we can estimate
\begin{equation}\label{ummmm}
\vert a_{i,j}(s,w)\vert \leq 2 \frac{\zeta(2\sigma)}{w^{i+j+2\sigma}} { 2 \vert s\vert +i+j-1 \choose i},
\end{equation}
for all $0\leq i,j < \infty$, where we have used
$$
\vert \zeta(2s+i+j) \vert \leq \zeta(2\sigma + i+j) \leq \zeta(2\sigma)
$$
with $\sigma = \mathrm{Re}(s)$. Recall from \eqref{bound_for_binomial_coefficients} that we can bound the binomial coefficient as
$$
{ 2 \vert s\vert +i+j-1 \choose i} \ll 2^{ 2 \vert s\vert + i + j }.
$$
Inserting this into \eqref{ummmm}, we obtain the bound
\begin{equation}\label{bound_for_aij}
\vert a_{i,j}(s,w)\vert  \ll \left( \frac{w}{2} \right)^{-(i+j)},
\end{equation}
with an implied constant depending only on $s$ and $w$ (but not on $i$ nor $j$). Now let $f\in H^2(\mathbb{D})$ be some function with a Taylor expansion as in \eqref{expansion}. Then by Proposition \ref{main_expression} and the definition of $\mathcal{A}_{s,w,k}$ we can write
\begin{equation}\label{difference}
(\mathcal{L}_{s,w} - \mathcal{A}_{s,w,k})f(z) = \sum_{ \substack{ i,j \in \mathbb{N}_0 \\ \max(i,j) \geq k }} a_{i,j}(s,w) c_j z^i.
\end{equation}
Recall that the functions $ \psi_{j} (z) = \sqrt{\frac{j+1}{\pi}}  z^j $ with $j\in \mathbb{N}_0$ provide an \textit{orthonormal} basis for $H^2(\mathbb{D}).$ It follows from \eqref{difference} that
\begin{equation}\label{applied_difference}
(\mathcal{L}_{s,w} - \mathcal{A}_{s,w,k})\psi_j = \sum_{ \substack{ i \in \mathbb{N}_0 \\ \max(i,j) \geq k }} a_{i,j}(s,w) \sqrt{\frac{j+1}{i+1}} \psi_i.
\end{equation}
Since $ \Vert \psi_i\Vert = 1 $, this gives  
$$
\Vert (\mathcal{L}_{s,w} - \mathcal{A}_{s,w,k})\psi_j \Vert \leq  \sum_{ \substack{ i \in \mathbb{N}_0 \\ \max(i,j) \geq k }} \vert a_{i,j}(s,w) \vert \sqrt{\frac{j+1}{i+1}}
$$
Using the bound in \eqref{bound_for_aij}, we obtain 
$$
\Vert (\mathcal{L}_{s,w} - \mathcal{A}_{s,w,k})\psi_j \Vert \ll \sum_{ \substack{ i \in \mathbb{N}_0 \\ \max(i,j) \geq k }} \left( \frac{w}{2} \right)^{-(i+j)}  \sqrt{\frac{j+1}{i+1}}
$$
Assuming first that $ j\geq k $, we can estimate this as
$$
\Vert (\mathcal{L}_{s,w} - \mathcal{A}_{s,w,k})\psi_j \Vert \ll \sum_{i\geq 1} \left( \frac{w}{2} \right)^{-(i+j)}  \sqrt{\frac{j+1}{i+1}} \ll \sqrt{j+1} \left( \frac{w}{2} \right)^{-j}.
$$
Similarly, assuming that $ j<k $, we have
$$
\Vert (\mathcal{L}_{s,w} - \mathcal{A}_{s,w,k})\psi_j \Vert \ll \sum_{i\geq k} \left( \frac{w}{2} \right)^{-(i+j)}  \sqrt{\frac{j+1}{i+1}} \ll \left( \frac{w}{2} \right)^{-k}
$$
Combining the two previous bounds, we can write
\begin{equation}\label{combination}
\Vert (\mathcal{L}_{s,w} - \mathcal{A}_{s,w,k})\psi_j \Vert  \leq \left( \frac{w}{2} \right)^{-\max(j,k) + o(j)}.
\end{equation}
Using the singular value estimate in \eqref{min_max_consequence} and the estimate in \eqref{combination}, we obtain for all $n\geq 1$ the estimate
\begin{align*}
\mu_n( \mathcal{L}_{s,w} - \mathcal{A}_{s,w,k} ) &\leq \sum_{j \geq n} \Vert (\mathcal{L}_{s,w} - \mathcal{A}_{s,w,k})\psi_j \Vert\\
&\leq \sum_{j \geq n} \Vert (\mathcal{L}_{s,w} - \mathcal{A}_{s,w,k})\psi_j \Vert\\
&\ll\sum_{j \geq n} \left( \frac{w}{2} \right)^{-\max(j,k) + o(j)}\\
&\ll \left( \frac{w}{2} \right)^{-\max(n,k) + o(k)}.
\end{align*}
Using this bound on singular values, we can finally estimate the trace norm as
\begin{align*}
\Vert \mathcal{L}_{s,w} - \mathcal{A}_{s,w,k} \Vert_1 &= \sum_{n=1}^\infty \mu_n( \mathcal{L}_{s,w} - \mathcal{A}_{s,w,k} )\\
&\ll \sum_{n=1}^\infty \left( \frac{w}{2} \right)^{-\max(n,k) + o(k)}\\
&\ll \left( \frac{w}{2} \right)^{-k + o(k)}.
\end{align*}
This finishes the proof of Lemma \ref{lemma:approx_finite_rank}.
\end{proof}

Using the Fredholm determinant identities from Theorem \ref{TO_main_theorem} and \eqref{equality_determinants} in conjuction with the bound on Fredholm determinants in \eqref{Simon_continuity_bound}, we obtain
\begin{align*}
\left\vert  Z_{\Gamma_w}(s) - D_k(s,w)  \right\vert &= \left\vert  \det\left( 1-\mathcal{L}_{s,w} \right) - \det\left( 1-\mathcal{A}_{s,w,k} \right) \right\vert\\
& \leq \Vert \mathcal{L}_{s,w} - \mathcal{A}_{s,w,k} \Vert_1  \exp\left(   \Vert \mathcal{L}_{s,w} \Vert_1 + \Vert \mathcal{A}_{s,w,k} \Vert_1 + 1 \right)\\
& \leq \Vert \mathcal{L}_{s,w} - \mathcal{A}_{s,w,k} \Vert_1  \exp\left(   2\Vert \mathcal{L}_{s,w} \Vert_1 + \Vert \mathcal{L}_{s,w} - \mathcal{A}_{s,w,k} \Vert_1 + 1 \right).
\end{align*}
Using Lemma \ref{lemma:approx_finite_rank} in the previous line gives
\begin{equation}\label{approx_alg}
\left\vert Z_{\Gamma_w}(s) - D_k(s,w) \right\vert \leq C  \left(\frac{w}{2}\right)^{-k+o(k)} \to 0
\end{equation}
as $k\to \infty$ for some constant $C = C(s,w) > 0$, proving the first part of Theorem \ref{approximation_by_matrices}.

To conclude the proof of Theorem \ref{approximation_by_matrices}, let $ \varepsilon > 0 $ be small enough so that $ \delta > \frac{1}{2}+\varepsilon. $ Using Rouch\'{e}'s theorem, the bound in \eqref{approx_alg}, and the fact that $ Z_{\Gamma_w}(s) $ has precisely one zero in $ \mathrm{Re}(s) \geq \frac{1}{2}+\varepsilon $, we can show that $ D_k(s,w) $ has exactly one zero in $ \mathrm{Re}(s) \geq \frac{1}{2} +\varepsilon $, provided $k$ is large enough. (In fact, we can use an argument similar to the final argument in our proof of Part \eqref{Part_3} of Theorem \ref{zeros_and_eigenvalues}.) Let $s_k(w)$ denote this zero. It is clear from \eqref{approx_alg} that $s_k(w)$ converges to $\delta$ as $k$ tends to infinity. In fact, we show

\begin{lemma}
Notations being as above, we have for $k$ sufficiently large
$$
\vert \delta - s_k(w) \vert \leq  C' \left( \frac{w}{2} \right)^{-k+o(k)}
$$
where $C' > 0$ is some constant depending only on $w$.
\end{lemma}

\begin{proof}
By the mean value theorem, there exists some $ t_k \in \mathbb{C} $ in the line segment joining $s_k(w)$ and $\delta$ such that
\begin{align*}
Z_{\Gamma_w}'(t_k)(\delta - s_k(w)) &= Z_{\Gamma_w}(\delta) - Z_{\Gamma_w}(s_k(w))\\
&= - Z_{\Gamma_w}(s_k(w))\\
&= D_k(s_k(w)) - Z_{\Gamma_w}(s_k(w)),
\end{align*}
and thus by \eqref{approx_alg} we get
\begin{equation}\label{consMeanValue}
\vert Z_{\Gamma_w}'(t_k) \vert \cdot \vert \delta - s_k(w) \vert \leq C \left( \frac{w}{2} \right)^{-k+o(k)}
\end{equation}
Now notice that $t_k$ must also converge to $\delta$ as $k\to \infty$ and in particular we have
$$
\lim_{k\to 0} Z_{\Gamma_w}'(t_k) = Z_{\Gamma_w}'(\delta).
$$
Note that $Z_{\Gamma_w}'(\delta) \neq 0$, since $s = \delta$ is a simple zero of $Z_{\Gamma_w}(s)$ by Part \eqref{Part_1} of Theorem \eqref{zeros_and_eigenvalues}. Thus, for all $k$ large enough we have
$$
\vert Z_{\Gamma_w}'(t_k) \vert \geq \frac{1}{2} \vert Z_{\Gamma_w}'(\delta)\vert > 0.
$$
Inserting this into \eqref{consMeanValue} we obtain
$$
\vert \delta - s_k(w) \vert \leq \frac{C}{\vert Z_{\Gamma_w}'(t_k) \vert} \left( \frac{w}{2} \right)^{-k+o(k)} \leq  \frac{2C}{\vert Z_{\Gamma_w}'(\delta) \vert } \left( \frac{w}{2} \right)^{-k+o(k)},
$$
completing the proof.
\end{proof}

\subsection{1-eigenfunctions of the transfer operator}
The results of this subsection are crucial for the proof of Theorem \ref{thm:asymptotic_expansion}. Recall that the Selberg zeta function $Z_{\Gamma_w}(s)$ vanishes at $ s = \delta $. We deduce from Theorem \ref{TO_main_theorem} and the general theory of Fredholm determinants that $ 1 $ is an eigenvalue of $ \mathcal{L}_{\delta,w} \colon H^2(\mathbb{D})\to H^2(\mathbb{D}) $. That is, there exists a non-zero function $ f \in H^{2}(\mathbb{D}) $ satisfying
\begin{equation}\label{eigen}
f(z) = \mathcal{L}_{\delta,w} f(z), \quad z\in \mathbb{D}.
\end{equation}
In this subsection we investigate the coefficients $c_i$ of the 1-eigenfunctions $f$ of $\mathcal{L}_{\delta,w}$ in the Taylor expansion 
\begin{equation}\label{expansion_2}
f(z) = \sum_{i=0}^{\infty} c_{i}z^{i}.
\end{equation}
The main result of this subsection is

\begin{prop}\label{IMPORTANT}
Assume that $ w\geq 3$ and let $f\neq 0$ be a 1-eigenfunction of $\mathcal{L}_{\delta,w}$. Then the Taylor-coefficients of $f$ in \eqref{expansion_2} have the following properties:
\begin{enumerate}[{\rm (i)}]
\setlength\itemsep{1em}
\item \label{Part_1_IMPORTANT} 
For all odd $ i $ we have $ c_i = 0$. In other words, $ f $ is an even function.

\item \label{Part_2_IMPORTANT}
For all even $ i\geq 2 $ we have the bound
$$
\vert c_i\vert \leq \frac{3 \zeta(3) }{\sqrt{\pi}} (i+1)\left( \frac{3}{2w}\right)^{i+1} \Vert f\Vert.
$$

\item \label{Part_3_IMPORTANT}
Moreover, the constant term of $ f $ satisfies
$$
\vert c_0\vert \geq  0.31 \Vert f\Vert.
$$
\end{enumerate}
\end{prop}

For the proof of Proposition \ref{IMPORTANT} we need some preparatory lemmas. Recall from \eqref{prelim_results} that $ 1/2 < \delta < 1$. We will occasionally use this estimate below without mention.

\begin{lemma}\label{forced_relation_lemma}
Assumptions and notations being as in Proposition \ref{IMPORTANT}, we have the relation
\begin{equation}\label{forced_relation}
c_i = \sum_{j=1}^\infty  a_{i,j}(\delta,w) c_j
\end{equation}
for every positive integer $i$, where $a_{i,j}(\delta, w)$ is given by Proposition \ref{main_expression}.
\end{lemma}

\begin{proof}
For every $1$-eigenfunction $f$ we have
\begin{equation}\label{main_expression_formala_2}
\sum_{i=0}^\infty c_i z^i = f(z) = \mathcal{L}_{\delta,w} f(z) = \sum_{i=0}^\infty \sum_{j=0}^\infty  a_{i,j}(\delta,w) c_j z^i.
\end{equation}
by Proposition \ref{main_expression}. Comparing the coefficients in \eqref{main_expression_formala_2} yields the relation in \eqref{forced_relation}.
\end{proof}

Also helpful is the following 

\begin{lemma}\label{ineq_for_power_series}
For all $ 0 < x < 1 $ and all positive integers $i$ we have 
$$
S_{i}^{\mathrm{even}}(x) \sceq \sum_{\substack{ j=0 \\ j\text{ even}}}^{\infty}  (j+1) { i+j+1 \choose i} x^{j} < \frac{i+1}{(1-x)^{i+2}}
$$
and
$$
S_{i}^{\mathrm{odd}}(x) \sceq \sum_{\substack{ j=1 \\ j \text{ odd}}}^{\infty} (j+1) { i+j+1 \choose i} x^{j} < \frac{i+1}{2(1-x)^{i+2}}. 
$$
\end{lemma}

\begin{proof}
It is an exercise to check that
$$
S_{i}(x) \sceq \sum_{j=0}^{\infty} (j+1) { i+j+1 \choose i} x^{j} = \frac{i+1}{(1-x)^{i+2}}
$$
for all $ \vert x\vert < 1. $ The result then follows from
$$
S_{i}^{\mathrm{even}}(x) = \frac{1}{2}\left( S_{i}(x) + S_{i}(-x) \right), \quad S_{i}^{\mathrm{odd}}(x) = \frac{1}{2}\left( S_{i}(x) - S_{i}(-x) \right)
$$
and from the fact that 
$$ S_{i}(x) > S_{i}(-x) > 0, $$ 
provided $x>0$.
\end{proof}

The next result will be needed to show that 1-eigenfunctions of $ \mathcal{L}_{\delta,w} $ are even functions.

\begin{lemma}\label{good_pairs}
Fix a 1-eigenfunction $f$ with Taylor expansion as in \eqref{expansion_2}. We call a pair of positive numbers $ (\alpha, \eta) $ `good' (for the 1-eigenfunction $f$) if  the bound
\begin{equation}\label{goodness}
\vert c_{i}\vert < \alpha (i+1) \left( \frac{\eta}{w} \right)^{i+1}.
\end{equation}
is satisfied for all odd $ i\geq 1 $. Then the following holds: if $ (\alpha, \eta) $ is a good pair with $ \eta w^{-2} < 1 $, then 
$$ \left( \frac{\alpha \eta}{2(1-\eta w^{-2})}, \frac{1}{1-\eta w^{-2}} \right) $$ 
is also a good pair. 
\end{lemma}

\begin{proof}
It follows directly from the expression in \eqref{main_expression_formala} that $a_{i,j}(\delta,w) = 0$ whenever $i$ and $j$ have different parity, and that
\begin{equation}\label{hoi}
\vert a_{i,j}(\delta,w)\vert \leq  \frac{2 \zeta(2\delta +i+ j)}{w^{i+j+2\delta}} { 2\delta+i+j-1 \choose i} \leq \frac{2 \zeta(3)}{w^{i+j+2\delta}} { i+j+1 \choose i}
\end{equation}
when $i$ and $j \geq 1$ have the same parity.

Assume that $ i\geq 1 $ is an odd integer. Then, using Lemma \ref{forced_relation_lemma} and the estimate in \eqref{hoi}, we obtain
\begin{align*}
\vert c_{i}\vert &\leq  \frac{2 \zeta(3)}{w^{i + 2 \delta}} \sum_{\substack{ j=1 \\ j\text{ odd } }}^{\infty} \vert c_j\vert  \frac{1}{w^{j}} { i+j+1 \choose i}
\end{align*}
Now assume that $ (\alpha, \eta)\in \mathbb{R}_{>0}^2 $ is a good pair with $ \eta w^{-2} < 1 $. Inserting \eqref{goodness} into the previous line and rearranging then gives
\begin{align*}
\vert c_{i}\vert &<\frac{2\alpha \zeta(3)}{w^{i+2\delta}} \sum_{\substack{ j=1 \\ j\text{ odd } }}^{\infty} (j+1)\left( \frac{\eta}{w} \right)^{j+1} \frac{1}{w^{j}} { i+j+1 \choose i}\\
&= \frac{2\alpha \eta \zeta(3) }{w^{i+2\delta+1}} \sum_{\substack{ j=1 \\ j\text{ odd } }}^{\infty} (j+1)  \left( \eta w^{-2} \right)^{j} { i+j +1 \choose i}\\
&= \frac{2\alpha \eta \zeta(3)}{w^{i+2\delta+1}} S^{\mathrm{odd}}_i\left( \eta w^{-2} \right)\\
&< \frac{\alpha \eta \zeta(3) }{w^{i+2\delta+1}} (i+1) \left( \frac{1}{1-\eta w^{-2}}\right)^{i+2},
\end{align*}
where in the last line we have used Lemma \ref{ineq_for_power_series}. Hence, we have shown
$$
\vert c_{i}\vert < \widetilde{\alpha} (i+1) \left( \frac{\widetilde{\eta}}{w} \right)^{i+1}  
$$
where 
$$
\widetilde{\alpha} = \frac{\alpha \eta \zeta(3)}{ w^{2\delta} (1-\eta w^{-2})} \quad \text{and}\quad \widetilde{\eta} = \frac{1}{1-\eta w^{-2}}.
$$
Noticing that 
$$ \frac{\zeta(3)}{w^{2\delta}} < \frac{\zeta(3)}{3} < \frac{1}{2}, $$
we obtain furthermore 
$$
\widetilde{\alpha} < \frac{ \alpha \eta}{2(1-\eta w^{-2})},
$$
completing the proof of Lemma \ref{good_pairs}.
\end{proof}

\begin{proof}[Proof of Proposition \ref{IMPORTANT}]
We may assume without loss of generality that the $1$-eigenfunction $ f $ is normalized so that $ \Vert f\Vert = 1 $. Recall from \eqref{trivial_c_k} that we have the a-priori bound on coefficients 
\begin{equation}\label{trivial_c_k_new}
\vert c_j\vert \leq \frac{\sqrt{j+1}}{\sqrt{\pi}} \leq \frac{j+1}{\sqrt{\pi}} \quad \text{for all} \quad j\geq 0.
\end{equation}
Let us first prove Part \eqref{Part_2_IMPORTANT} and assume that $ i\geq 2 $ is even. Note that in this case we have $ a_{i,j}(\delta, w) = 0 $ whenever $j$ is odd. Thus we can use Lemma \ref{forced_relation_lemma} together with the bounds in \eqref{trivial_c_k_new} and in \eqref{hoi}, to estimate 
\begin{align*}
\vert c_{i}\vert &\leq \sum_{\substack{ j=0 \\ j\text{ even} }}^\infty  \vert a_{i,j}(\delta,w) \vert \vert c_j \vert\\
&\leq \frac{2  \zeta(3) }{w^{i+2\delta} \sqrt{\pi}} \sum_{\substack{ j=0 \\ j\text{ even} }}^{\infty} (j+1)  { i+j+ 1 \choose i} \frac{1}{w^j}\\
&= \frac{2 \zeta(3)}{w^{i+2\delta} \sqrt{\pi}} S^{\mathrm{even}}_i(w^{-1}).
\end{align*}
Recalling that $w \geq 3$ we can use Lemma \ref{ineq_for_power_series} to obtain furthermore
\begin{equation}
\vert c_{i}\vert  \leq \frac{2 \zeta(3) }{w^{i+1} \sqrt{\pi}} (i+1)\left( \frac{3}{2}\right)^{i+2} = \frac{3 \zeta(3) }{\sqrt{\pi}} (i+1)\left( \frac{3}{2w}\right)^{i+1},
\end{equation}
which completes the proof of \eqref{Part_2_IMPORTANT}. 

Let us now prove Part \eqref{Part_1_IMPORTANT} and address the case when $ i\geq 1 $ is an odd integer. By repeating the same steps as above, we obtain an estimate of the type
$$
\vert c_{i}\vert < \alpha (i+1)\left( \frac{3}{2w} \right)^{i+1}.  
$$
for all odd $ i\geq 1 $ where $ \alpha > 0 $ is some absolute constant. In the language of Lemma \ref{good_pairs} this means that the pair $ (\alpha_0, \eta_0):= \left( \alpha, \frac{3}{2}\right) $ is good.  By iterating Lemma \ref{good_pairs} we obtain a sequence of good pairs $ (\alpha_\ell, \eta_\ell) $ recursively defined by
$$
\alpha_{\ell} = \alpha_{\ell - 1} \frac{\eta_{\ell - 1}}{2\left( 1-\eta_{\ell - 1} w^{-2} \right)} \quad \text{and}\quad  \eta_{\ell} = \frac{1}{1-\eta_{\ell - 1} w^{-2}}.
$$
One can check that the sequence $ \eta_\ell $ is decreasing as $ \ell \to \infty $, so $ \eta_\ell \leq \eta_0 = \frac{3}{2}. $ Moreover, since 
$$ x \mapsto \frac{x}{2\left( 1-x w^{-2} \right)} $$ 
is an increasing function, we get
$$
\alpha_{\ell} = \alpha_{\ell - 1} \frac{ \eta_{\ell - 1}}{2\left( 1-\eta_{\ell - 1} w^{-2} \right)} \leq \alpha_{\ell - 1} \frac{ \eta_{0}}{2\left( 1-\eta_{0} w^{-2} \right)} \leq 0.9 \alpha_{\ell - 1},
$$
where for the last inequality we used the assumption that $w\geq 3$. This implies that
$$
\alpha_{\ell} \leq 0.9^{\ell}\alpha \to 0 \quad (\ell\to \infty),
$$
which in turn implies that for all odd $i\geq 1$ we have
$$
\vert c_i\vert < \alpha_\ell (i+1) \left( \frac{\eta_\ell}{w}\right)^{i+1} \leq \alpha_\ell (i+1) \left( \frac{\eta_0}{w}\right)^{i+1} \to 0  \quad (\ell\to \infty).
$$
But this forces $ c_i = 0 $ for all odd $i \geq 1$, completing the proof of Part \eqref{Part_1_IMPORTANT}. 

To prove Part \eqref{Part_3_IMPORTANT} recall from \eqref{norm} that the norm of $f$ can be expressed in term of its Taylor coefficients as
\begin{equation}
\Vert f\Vert^{2} = \pi \sum_{j=0}^{\infty}  \frac{\vert c_j\vert^2}{j+1}.
\end{equation}
Since $ c_j = 0 $ for all odd $ j $, we can restrict this sum to the even terms and isolate the 0-th term, writing
$$
\Vert f\Vert^{2} = \pi \vert c_0\vert^2 + \pi  \sum_{l=1}^{\infty}  \frac{\vert c_{2l}\vert^2}{2l+1}. 
$$
By assumption we have $ \Vert f\Vert = 1 $ and $w\geq 3$. Hence, using the bound on coefficients in Part \eqref{Part_2_IMPORTANT}, this gives
\begin{align*}
1 &\leq \pi \vert c_0\vert^2 + 9 \zeta(3)^{2}  \sum_{l=1}^{\infty}  (2l+1) \left( \frac{3}{2w}\right)^{4l+2}\\
&\leq \pi \vert c_0\vert^2 + 9 \zeta(3)^{2}  \sum_{l=1}^{\infty}  (2l+1) \left( \frac{1}{2}\right)^{4l+2}\\
&\leq \pi \vert c_0\vert^2 + 0.68.
\end{align*}
Rearranging this inequality, we obtain
$$
\vert c_0\vert \geq \sqrt{\frac{0.32}{\pi}} > 0.31,
$$
completing the proof of Proposition \ref{IMPORTANT}.
\end{proof}

\subsection{Finishing the proof of Theorem \ref{thm:asymptotic_expansion}}
We can now prove Theorem \ref{thm:asymptotic_expansion}. Let $ f \in H^2(\mathbb{D}) $ be a non-zero 1-eigenfunction of $\mathcal{L}_{\delta,w}$ with Taylor-expansion
$$
f(z) = c_0 + c_1 z + c_2 z^2 + \cdots.
$$
We may assume without loss of generality that $f$ is normalized so that $\Vert f\Vert = 1.$ Applying Lemma \ref{forced_relation_lemma} with $ i=0 $ gives
$$
c_{0} = \sum_{j=0}^{\infty} \alpha_{0,j}(\delta, w)  c_j  = \sum_{j=0}^{\infty} c_j \left( (-1)^j + 1 \right) \frac{1}{w^{2\delta+j}} \zeta(2\delta+j).
$$
Notice that we can restrict this sum to even terms $j=2l$ and isolate the term $ l=0 $ to write
\begin{equation}\label{zeroth_equation}
c_{0} = 2 \zeta(2\delta) \frac{1}{w^{2\delta}} c_{0} + 2 \sum_{l=1}^{\infty} c_{2l} \frac{1}{w^{2l+2\delta}} \zeta(2l+2\delta).
\end{equation}
We are interested in the behavior of $ \delta = \delta(w) $ as $ w\to \infty $, so we may assume that $ w\geq 3 $. Then, by Part \eqref{Part_3_IMPORTANT} of Proposition \ref{IMPORTANT}, we have 
$$ c_0 \neq 0. $$ 
Thus we can divide both sides of \eqref{zeroth_equation} by $ c_0 $ to obtain
\begin{equation}\label{important}
1 = 2 \zeta(2\delta) \frac{1}{w^{2\delta}} + E(w),
\end{equation}
where we have put
$$
E(w) = 2c_{0}^{-1} \sum_{l=1}^{\infty} c_{2l} \frac{1}{w^{2l+2\delta}} \zeta(2l+2\delta). 
$$
Invoking the estimates in Parts \eqref{Part_2_IMPORTANT} and \eqref{Part_3_IMPORTANT} of Proposition \ref{IMPORTANT}, and recalling that $ \delta > \frac{1}{2} $, we get the bound
\begin{equation}
\vert E( w)\vert \leq 2 \vert c_{0}\vert^{-1} \sum_{l=1}^{\infty} \vert c_{2l}\vert \frac{1}{w^{2l+2\delta}} \zeta(2l+2\delta)=
O\left( \sum_{l=1}^{\infty} \frac{1}{w^{4l+2\delta+1}}\right) =  O\left( \frac{1}{w^6} \right)
\end{equation}
where the implied constant in the error term does not depend on $ w $. Thus, returning to \eqref{important}, we have
\begin{equation}\label{equation_with_unknown}
1 = 2 \zeta(2\delta) \frac{1}{w^{2\delta}}  +  O\left( \frac{1}{w^6} \right).
\end{equation}
The final step towards the proof of Theorem \ref{thm:asymptotic_expansion} is to `solve' this equation for the unknown variable $ \delta $. On introducing a new variable $x>0$ and making the substitution
$$
\delta = \frac{1+x}{2},
$$
we can rewrite \eqref{equation_with_unknown} as
\begin{equation}\label{lady_gaga}
1 = 2 \zeta(1 + x)  \frac{1}{w^{1+x}} + O\left( \frac{1}{w^6} \right).
\end{equation}
Recalling the well-known Laurent expansion of the Riemann zeta function $\zeta(s)$ at $ s = 1 $, we write
\begin{equation}\label{laurent_expansion_riemann}
\zeta(1+x) =  \frac{1}{x} + \sum_{n=0}^\infty \frac{(-1)^n \gamma_n}{n!} x^{n},
\end{equation}
where $ \gamma_n $ is the $ n $-th \textit{Stieltjes constant}\footnote{The 0-th Stieltjes constant $ \gamma_0 \approx 0,57721 56649 $ is better known as the `Euler--Mascheroni constant'. Numerical approximations for the Stieltjes constants can be found in OEIS \cite{OEIS}.}. Notice also that we can write
\begin{equation}\label{log_expansion}
\frac{1}{w^{1+x}}  = \frac{e^{-x t}}{w} = \frac{1}{w} \sum_{n=0}^{\infty} \frac{(-t)^{n}}{n!}x^{n}.
\end{equation}
where we have set 
$$t=\log w$$
for notational convenience. Using the Cauchy product formula, we can multiply the series expansions in \eqref{laurent_expansion_riemann} and \eqref{log_expansion} to obtain an expression of the form
\begin{equation}\label{expre_after_Cauchy}
2 \zeta(1 + x)  \frac{1}{w^{1+x}} = \frac{2}{xw} + \frac{2}{w}\sum_{n=0}^\infty Q_{n+1}(t) x^n,
\end{equation}
where each $Q_n$ is a polynomial of degree $n$ whose coefficients can be computed in terms of the Stieltjes constants. Notice in particular that
$$
Q_1(t) = - t + \gamma_0.
$$
We can truncate the series on the right of \eqref{expre_after_Cauchy} at $n=4$ to write
\begin{equation}\label{expre_after_Cauchy_2}
2 \zeta(1 + x)  \frac{1}{w^{1+x}} = \frac{2}{xw} + Q_1(t) \frac{2}{w} + Q_2(t) \frac{2x}{w}+ Q_3(t) \frac{2x^2}{w} + Q_4(t) \frac{2x^3}{w} + Q_4(t) \frac{2x^4}{w} + O \left( \frac{\log(w)^5}{w^5} \right).
\end{equation}
Thus, going back to \eqref{lady_gaga}, we have shown that $x$ must satisfy
\begin{equation}\label{almost_done}
1 = \frac{2}{xw} + Q_1(t) \frac{2}{w} + Q_2(t) \frac{2x}{w}+ Q_3(t) \frac{2x^2}{w} + Q_4(t) \frac{2x^3}{w} + Q_5(t) \frac{2x^4}{w} + O\left(\frac{\log(w)^5}{w^5}\right).
\end{equation}
Notice that the term $ O( \frac{1}{w^6}) $ we had obtained in \eqref{lady_gaga} gets absorbed by the term $ O( \frac{(\log w)^5}{w^5}) $ on the right hand side of \eqref{expre_after_Cauchy_2}.

Now, multiplying both sides by \eqref{almost_done} by $x$ yields
\begin{equation}\label{almost_done_2}
x = \frac{2}{w} + Q_1(t) \frac{2x}{w} + Q_2(t) \frac{2x^2}{w}+ Q_3(t) \frac{2x^3}{w} + Q_4(t) \frac{2x^4}{w} + Q_5(t) \frac{2x^5}{w} + O\left(x\frac{\log(w)^5}{w^5}\right).
\end{equation}
Recall from \eqref{prelim_results} that $\delta = \delta(w) \to \frac{1}{2}^+$ which implies that $x\to 0^+$. Thus, \eqref{almost_done_2} immediately implies the a-priori bound
$$
x = O\left( \frac{1}{w} \right)
$$
as $w\to \infty$. Inserting this bound into the error term in \eqref{almost_done_2} gives
\begin{equation}\label{interject}
x = \frac{2}{w} + Q_1(t) \frac{2x}{w} + Q_2(t) \frac{2x^2}{w}+ Q_3(t) \frac{2x^3}{w} + Q_4(t) \frac{2x^4}{w} + O\left(\frac{\log(w)^5}{w^6}\right).
\end{equation}
We can now repeatedly substitute every occurrence of $x$ on the right hand side of this expression by the expression itself, leading to an expression of the form
\begin{equation}\label{almost_done_3}
x = \frac{2}{w} + P_1(t) \frac{2}{w^2} + P_2(t) \frac{2}{w^3}+  P_3(t) \frac{2}{w^4} +  P_4(t) \frac{2}{w^5}  + O\left(\frac{\log(w)^5}{w^6}\right).
\end{equation}
where $P_1, P_2, P_3, P_4$ can be determined explicitly from the polynomials $Q_1, Q_2, Q_3, Q_4$ (this shows in particular that each $ P_i $ is a polynomial of degree at most $i$ and that its coefficients can be computed in terms of the Stieltjes constants). In particular, this procedure yields (after the first substitution)
$$
P_1(t) = 2 Q_1(t) = -2 t + 2\gamma_0.
$$
By re-substituting the variables, we obtain 
$$
\delta = \frac{1+x}{2} = \frac{1}{2} + \frac{1}{w} + P_1(t) \frac{1}{w^2} + P_2(t) \frac{1}{w^3}+  P_3(t) \frac{1}{w^3} +  P_4(t) \frac{1}{w^5}  + O\left(\frac{\log(w)^5}{w^6}\right),
$$
as $w\to \infty$, completing the proof of Theorem \ref{thm:asymptotic_expansion}.

\subsection{Sharp numerical estimates}\label{Sharp numerical estimates}
In this subsection we show how to obtain numerical estimates for $ \delta(w) $.
The case $ w = 3 $ will be of special interest (due to the question posed by Jakobson--Naud in \cite{JN_lower_bounds}), but we will initially work with arbitrary $ w\geq 3 $ and write $ \delta =\delta(w) $. 

We fix a non-zero 1-eigenfunction $ f(z) = c_0 + c_1 z + \cdots $ a of $ \mathcal{L}_{\delta, w} $ and we assume that $ f $ is normalized so that $ \Vert f\Vert = 1. $ Specializing Lemma \ref{forced_relation_lemma} to $ i=0 $ gives
\begin{equation}\label{AAA}
c_{0} =  2\sum_{l=0}^{\infty} c_{2l}  \frac{1}{w^{2l+2\delta}}  \zeta(2l+2\delta)
=  2c_{0}  \frac{1}{w^{2\delta}}  \zeta(2\delta) + 2c_{2}  \frac{1}{w^{2+2\delta}}  \zeta(2+2\delta) + \mathcal{E}_1(w)
\end{equation}
where
$$
\mathcal{E}_1(w) \sceq 2 \sum_{l=2}^{\infty} c_{2l} \frac{1}{w^{2l+2\delta}}\zeta(2l+2\delta)
$$
Similarly, specializing Lemma \ref{forced_relation_lemma} to $ i=2 $ yields
\begin{equation}\label{BBB}
c_{2} = 2 c_{0}  \frac{1}{w^{2\delta+2}} { 2\delta + 1 \choose 2} \zeta(2+2\delta) + 2 c_{2}  \frac{1}{w^{4+2\delta}} { 2\delta + 3 \choose 2} \zeta(4+2\delta)+ \mathcal{E}_2(w),
\end{equation}
where
$$
\mathcal{E}_2(w)\sceq 2 \frac{1}{w^{2}} \sum_{l=2}^{\infty} c_{2l}  \frac{1}{w^{2l+2\delta}} { 2l + 2\delta + 1 \choose 2} \zeta(2+2l+2\delta)
$$
Solving \eqref{BBB} for $ c_2 $ yields 
\begin{equation}\label{c_2}
c_2 = c_{0} \frac{ \frac{2}{w^{2+2\delta}} { 2\delta + 1 \choose 2} \zeta(2+2\delta)}{1- \frac{2}{w^{4+2\delta}} { 2\delta + 3 \choose 2} \zeta(4+2\delta)} + \frac{\mathcal{E}_2(w)}{1- \frac{2}{w^{4+2\delta}} { 2\delta + 3 \choose 2} \zeta(4+2\delta)}
\end{equation}
Inserting \eqref{c_2} into \eqref{AAA} gives
\begin{equation}\label{refined_equation}
c_0 = c_{0} \left(   \frac{2}{w^{2\delta}}  \zeta(2\delta) + \frac{ \frac{4}{w^{4+4\delta}} { 2\delta + 1 \choose 2} \zeta(2+2\delta)^2}{1- \frac{2}{w^{4+2\delta}} { 2\delta + 3 \choose 2} \zeta(4+2\delta)}    \right) + \mathcal{E}(w)
\end{equation}
with 
\begin{equation}\label{ERRRR}
\mathcal{E}(w) \sceq \mathcal{E}_1(w) + \frac{\frac{2}{w^{2+2\delta}}  \zeta(2+2\delta)}{1- \frac{2}{w^{4+2\delta}} { 2\delta + 3 \choose 2} \zeta(4+2\delta)} \mathcal{E}_2(w).
\end{equation}
Recall from Proposition \ref{IMPORTANT} that $ c_0\neq 0 $. Hence, we can divide both sides of \eqref{refined_equation} by $ c_0 $ to obtain 
\begin{equation}\label{to_solve}
\left\vert 1-  \frac{2}{w^{2\delta}} \zeta(2\delta) - \frac{\frac{4}{w^{4+4\delta}} { 2\delta + 1 \choose 2} \zeta(2+2\delta)^2}{1- \frac{2}{w^{4+2\delta}} { 2\delta + 3 \choose 2} \zeta(4+2\delta)}   \right\vert  =  E(w),
\end{equation}
where
$$
E(w)\sceq \frac{\vert \mathcal{E}(w)\vert}{\vert c_0\vert}
$$
The subsequent goal is to estimate the error $ E(w). $ By Part \eqref{Part_3_IMPORTANT} of Proposition \ref{IMPORTANT} we have
\begin{equation}\label{newer}
E(w) \leq \frac{1}{0.31}\vert \mathcal{E}(w)\vert,
\end{equation}
so we have to estimate $\vert \mathcal{E}(w)\vert$. Recall that we have the coefficient bound
$$
\vert c_i\vert < \frac{3 \zeta(3) }{\sqrt{\pi}} (i+1)\left( \frac{3}{2w}\right)^{i+1}.
$$
from Proposition \ref{IMPORTANT}. Thus
\begin{align*}
\vert \mathcal{E}_1(w)\vert &\leq 2 \sum_{l=2}^{\infty} \vert c_{2l}\vert  \frac{1}{w^{2l+2\delta}}  \zeta(2l+2\delta)\\
&< \frac{9 \zeta(3) \zeta(4+2\delta)}{w \sqrt{\pi}} \sum_{l=2}^{\infty}  (2l+1)\left( \frac{3}{2w}\right)^{2l} \frac{1}{w^{2l+2\delta}}  \\
&= \frac{9 \zeta(3) \zeta(4+2\delta) }{w^{1+2\delta} \sqrt{\pi}} \sum_{l=2}^{\infty}  (2l+1)\left( \frac{3}{2w^2}\right)^{2l}\\
&=  \frac{9\zeta(3) \zeta(4+2\delta) }{w^{1+2\delta} \sqrt{\pi}} \cdot \frac{(\frac{3}{2w^2})^{4} \left( 5-3 (\frac{3}{2w^2})^{2}  \right)  }{\left( 1- (\frac{3}{2w^2})^{2} \right)^2}\\
&= \frac{729 \, \zeta(3) \zeta(4+2\delta) \left( 5-3 (\frac{3}{2w^2})^{2}  \right) }{16\, w^{9+2\delta} \sqrt{\pi} \left( 1- (\frac{3}{2w^2})^{2} \right)^2 }. \numberthis\label{E1}
\end{align*}
In the second last line we have used the elementary identity 
\begin{equation}\label{elementary_closed_formula}
\sum_{l=2}^{\infty}  (2l+1)x^{2l} = \frac{d}{dx} \sum_{l=2}^{\infty} x^{2l+1} = \frac{d}{dx} \left( \frac{x^5}{1-x^2} \right) =  \frac{x^4 (5-3x^2)}{(1-x^2)^2}
\end{equation}
for $ x = \frac{3}{2w^2}. $ Similarly, we have
\begin{align*}
\vert \mathcal{E}_2(w)\vert &\leq 2 \frac{1}{w^{2}} \sum_{l=2}^{\infty} \vert c_{2l}\vert  \frac{1}{w^{2l+2\delta}} { 2l + 2\delta + 1 \choose 2} \zeta(2+2l+2\delta)\\
&< \frac{9  \zeta(3) \zeta(6+2\delta) }{ w^{3+2\delta} \sqrt{\pi}}  \sum_{l=2}^{\infty}  (2l+1)\left( \frac{3 }{2w^2}\right)^{2l} { 2l + 2\delta + 1 \choose 2} \\
&< \frac{9 \zeta(3) \zeta(7) }{ w^{3+2\delta} \sqrt{\pi}}  \sum_{l=2}^{\infty}  (2l+1)\left( \frac{3}{2w^2}\right)^{2l} { 2l + 3 \choose 2}.
\end{align*}
To estimate the remaining sum in the last line, we can use the identity
$$
\sum_{l=2}^{\infty}  (2l+1)x^{2l} { 2l + 3 \choose 2} = \frac{1}{2} \frac{d^3}{d x^3} \left( \frac{x^7}{1-x^2}  \right)  = 3x^{4} \cdot\frac{35-56x^2+39x^4-10x^6}{(1-x^{2})^{4}}.
$$
One can then check that
$$
\sum_{l=2}^{\infty}  (2l+1)x^{2l} { 2l + 3 \choose 2} < 113 x^{4} \quad \text{for all}\quad x\leq 1/6.
$$
Inserting this bound above, we obtain for all $ w\geq 3 $ the somewhat simpler estimate
\begin{equation}\label{E2}
\vert \mathcal{E}_2(w)\vert < \frac{9 \zeta(3) \zeta(7) }{ w^{3+2\delta} \sqrt{\pi}}\cdot 113 \left( \frac{3 }{2w^2}\right)^{4} <  \frac{3521}{w^{11+2\delta}}.
\end{equation}
Going back to \eqref{ERRRR} and gathering the estimates in \eqref{newer}, \eqref{E1}, \eqref{E2}, we obtain the following final bound for the error:
\begin{equation}\label{final_bound_error}
\vert E(w)\vert < \frac{1}{0.31}\left( \frac{729 \, \zeta(3) \zeta(4+2\delta) \left( 5-3 (\frac{3}{2w^2})^{2}  \right) }{16\, w^{9+2\delta} \sqrt{\pi} \left( 1- (\frac{3}{2w^2})^{2} \right)^2 } + \frac{7042\, \zeta(2+2\delta)}{w^{13+4\delta}\left(   1- 2 \frac{1}{w^{4+2\delta}} { 2\delta + 3 \choose 2} \zeta(4+2\delta) \right)} \right).
\end{equation}
Let us now specialize to the case $ w = 3.$ To estimate the error term we may use the already established numerical estimates by Phillips--Sarnak in \cite{Phillips_Sarnak_Hecke_groups}. We will simply use the (weaker) lower bound $ \delta = \delta(3) > 0.7 $. Note that the right hand side of \eqref{final_bound_error} is decreasing as a function of $\delta$, so we can insert these values to obtain
$$
\vert E(3)\vert < 0.0066.
$$
Thus, going back to \eqref{to_solve}, we deduce that $ \delta = \delta(3) $ must satisfy
\begin{equation}\label{lastEqToSolve}
\left\vert 1- 2  (1/3)^{2\delta}  \zeta(2\delta) - \frac{ 4 (1/3)^{4+4\delta} { 2\delta + 1 \choose 2} \zeta(2+2\delta)^2}{1- 2  (1/3)^{4+2\delta} { 2\delta + 3 \choose 2} \zeta(4+2\delta)}   \right\vert < \varepsilon \sceq 0.0066.
\end{equation}
The function 
$$
F\colon \left( \frac{1}{2}, \infty  \right) \to \mathbb{R}, \quad F(\delta) = 1- 2  (1/3)^{2\delta}  \zeta(2\delta) - \frac{ 4 (1/3)^{4+4\delta} { 2\delta + 1 \choose 2} \zeta(2+2\delta)^2}{1- 2  (1/3)^{4+2\delta} { 2\delta + 3 \choose 2} \zeta(4+2\delta)} 
$$
is strictly increasing, so \eqref{lastEqToSolve} forces $ \delta(3) $ to lie in the range
$$
\delta^{-} < \delta(3) < \delta^{+},
$$
where $ \delta^{\pm} \in (\frac{1}{2}, \infty) $ are the unique solutions of
$$
F( \delta^{\pm} ) = \pm \varepsilon.
$$
We can now check that
$$
F( 0.75065 ) < -\varepsilon \quad \text{and} \quad F( 0.75322 ) > \varepsilon,
$$
showing that
$$
0.75065 < \delta(3) < 0.75322.
$$

\normalem
\bibliography{ap_bib} 
\bibliographystyle{amsplain}

\end{document}